\documentclass[11pt, psamsfonts, reqno]{amsart}
\usepackage[T1]{fontenc}
\usepackage[latin1]{inputenc}
\usepackage[frenchb, english]{babel}
\usepackage{amsthm}
\usepackage{amssymb,amsmath,bbm,dsfont}
\usepackage{enumitem}
\usepackage[all]{xy}
\usepackage{mathrsfs}
\usepackage[left=0.0cm,right=0.0cm,top=1.7cm,bottom=1.7cm]{geometry}
\usepackage{hyperref}
\usepackage{verbatim}

\theoremstyle{plain}
\newtheorem{theorem}{Theorem}[section]

\newtheorem{lemma}[theorem]{Lemma}
\newtheorem{proposition}[theorem]{Proposition}
\newtheorem{prop}[theorem]{Proposition}
\newtheorem{corollary}[theorem]{Corollary}
\newtheorem{cor}[theorem]{Corollary}

\theoremstyle{definition}
\newtheorem{definition}[theorem]{Definition}

\theoremstyle{remark}
\newtheorem{remark}[theorem]{Remark}

\theoremstyle{plain}
\newtheorem{thma}{Theorem}

\newtheorem{thmb}[thma]{Theorem}

\def\Q{{\bf Q}}

\def\Z{{\bf Z}}
\def\C{{\bf C}}

\def\R{{\bf R}}

\def\O{{\mathcal{O}}}

\def\Am{{\mathcal{A}}}

\def\Ff{{\mathcal{F}}}

\def\A{{\bf A}}
\def\Af{{\bf A}_{f}}

\def\Eis{\operatorname{_cEis}}

\def\epsilon{\varepsilon}

\def\GSp{{\mathrm{GSp}}}
\def\GL{\mathrm{GL}}
\def\Gm{\mathbb{G}_{\rm m}}

\def\et{{\rm \acute{e}t}}

\title{Norm-compatible systems of cohomology classes for $\operatorname{GU(2,2)}$}

\author{Antonio Cauchi }
\thanks{* The author was supported by the Engineering and Physical Sciences Research Council [EP/L015234/1]. The EPSRC Centre for Doctoral Training in Geometry and Number Theory (The London School of Geometry and Number Theory), University College London.}

\makeatletter
\def\@tocline#1#2#3#4#5#6#7{\relax
  \ifnum #1>\c@tocdepth 
  \else
    \par \addpenalty\@secpenalty\addvspace{#2}%
    \begingroup \hyphenpenalty\@M
    \@ifempty{#4}{%
      \@tempdima\csname r@tocindent\number#1\endcsname\relax
    }{%
      \@tempdima#4\relax
    }%
    \parindent\z@ \leftskip#3\relax \advance\leftskip\@tempdima\relax
    \rightskip\@pnumwidth plus4em \parfillskip-\@pnumwidth
    #5\leavevmode\hskip-\@tempdima
      \ifcase #1
       \or\or \hskip 1em \or \hskip 2em \else \hskip 3em \fi%
      #6\nobreak\relax
    \dotfill\hbox to\@pnumwidth{\@tocpagenum{#7}}\par
    \nobreak
    \endgroup
  \fi}
\makeatother
\usepackage{hyperref}

\begin{document}

\begin{abstract} 
We describe work of Faltings on the construction of  \'etale  cohomology classes associated to symplectic Shimura varieties and show that they satisfy certain trace compatibilities similar to the ones of Siegel units in the modular curve case. Starting from those, we construct a two variable family of trace compatible classes in the cohomology of a unitary Shimura variety. \end{abstract}

\maketitle
\selectlanguage{english}

\setcounter{tocdepth}{1}
\tableofcontents
\section{Introduction}

\subsection{Motivation.} The explicit construction of motivic and \'etale cohomology classes in the cohomology of modular curves has offered a beautiful approach to study cases of the Birch and Swinnerton-Dyer conjecture (e.g. \cite{kolyvagin}, \cite{kato}, \cite{BDR2}, \cite{KLZ}, or \cite{DRdiagonal}), the Iwasawa Main Conjecture for modular forms (e.g. \cite{kato}, \cite{KLZ}), and it constitutes one of the main tools to study the arithmetic of Galois representations appearing in the \'etale cohomology of Shimura varieties and their relation to special $L$-values. Siegel units and their higher weight analogues, so-called Eisenstein classes, are a fascinating source of such cohomology classes. For instance, these motivic objects and their \'etale incarnation crucially appear in the constructions of Kato's Euler System (\cite{kato}), Beilinson-Flach Euler system (\cite{LLZ1}), and the Euler system for $\GSp_4$ (\cite{LSZ1}). \\  

One of the key features of Eisenstein classes is their $p$-adic variation properties, which allows us to construct integral \'etale cohomology classes, and it is used to vary the above mentioned Euler systems in Hida families. This is indeed an essential ingredient in the proof of an explicit reciprocity law in the Beilinson-Flach case (\cite{KLZ}). The theory of $p$-adic variation of Eisenstein classes was initiated by Kings in \cite{ellipticsoule}, who systematically studied the construction of $\Lambda$-adic Eisenstein classes, that are \'etale cohomology classes with coefficients in sheaves of Iwasawa algebras which interpolate the \'etale Eisenstein classes. Despite the substantial difference between the construction of Siegel units and the higher weight motivic Eisenstein classes, Kings showed how $\Lambda$-adic Eisenstein classes can be reconstructed from the \'etale realisation of Siegel units, using Soul\'e's twisting construction previously developed in the cyclotomic case. This phenomenon also appears in the construction of Kato's Euler system. \newline 

It is natural to ask whether constructions of Eisenstein classes for abelian schemes enjoy similar properties and have interesting arithmetic applications. At present, little is known in this direction. 
In \cite{padicinterpolation}, Kings carried out the study of the $p$-adic variation of the \'etale Eisenstein classes associated to general commutative group schemes and constructed $\Lambda$-Eisenstein classes in this setting. On the other hand, in \cite{Faltings2005} Faltings constructed motivic and \'etale cohomology classes associated to abelian schemes, and he explored some of their properties in the case where the abelian scheme is the universal abelian scheme associated to the Shimura variety of $\GSp_4$.

\subsection{Main results.} In this article, we study $p$-adic interpolation properties of the \'etale Eisenstein classes defined by Faltings (\cite{Faltings2005}). \newline Let $Y_{\GSp_{2g}}(U)_{/\Q}$ be the $\GSp_{2g}$ Shimura variety of level $U$, which is a moduli space of principally polarised abelian schemes of relative dimension $g$ and  level structure U. For every integer $N > 1$, consider a sufficiently small open compact subgroup $U_{N} \subset \GSp_{2g}(\widehat{\Z})$ whose component at each prime $\ell \mid N$ is given by \[U_1(\ell^{v_\ell(N)}):=\{ M \in  \GSp_{2g}(\Z_\ell) | R_{2g}(M) \equiv (0, \cdots, 0,1) \text{ mod }\ell^{v_\ell(N)} \}, \]  where $R_i(M)$ denotes the $i$-th row of $M$. Fix a prime $p$. We use the construction of \'etale Eisenstein classes in \cite{Faltings2005} to define \[ \Eis_{m,N}^g \in H^{2g-1}_{\et}(Y_{\GSp_{2g}}(U_{N}),\Z_{p}(g)),\] depending on the choice of auxiliary integers $c$ and $m$ coprime with $p$, and we show the following trace compatibility relations.  
\begin{thma}[Proposition \ref{tracecompfalt}]\label{a} Let $\ell,N$ be coprime with $c$, and let $\operatorname{Tr}_{\pi_\ell}$ denote the trace map associated to the natural degeneracy map \[ \pi_\ell:Y_{\GSp_{2g}}(U_{N\ell}) \longrightarrow Y_{\GSp_{2g}}(U_{N}). \] We have \[\operatorname{Tr}_{\pi_\ell}(\Eis_{m,N\ell}^g) = \begin{cases} \Eis_{m,N}^g  & \text{ if }\ell\mid N; \\ 
\Eis_{m,N}^g- (\begin{smallmatrix} r I & 0 \\ 0 & r I \end{smallmatrix})^*(\Eis_{m,N}^g) & \text{ if } \ell \nmid N, \end{cases}\]
where $r$ denotes the inverse of $N$ modulo $\ell$.
\end{thma}
We use a method similar to the one used by Scholl in the proof of the $g=1$ case (\cite{scholl}). \\

As an application, we define classes in the cohomology of the unitary Shimura variety $Y_{\operatorname{GU}(2,2)}$ by pushing forward Faltings' classes associated to $Y_{\GSp_{4}}$ along (twisted) closed immersions \[ Y_{\GSp_{4}}(uV_{N,M}u^{-1} \cap \GSp_{4}) \hookrightarrow Y_{\operatorname{GU}(2,2)}(V_{N,M}), \] for suitable choices of $u \in \operatorname{GU(2,2)}(\Af)$ and level structure $V_{N,M} \subset \operatorname{GU(2,2)}(\widehat{\Z})$, whose component at each $\ell \mid N$ has pullback to $\GSp_4(\Z_\ell)$ contained in (and equal to if $\ell \nmid M$) $U_1(\ell^{v_\ell(N)})$.  We show that the obtained element \[{_c\mathscr{Z}_{N,M,u}} \in H^{5}_{\et}(Y_{\operatorname{GU}(2,2)}(V_{N,M}),\Z_{p}(3))\] satisfies the following (Proposition \ref{klingenunitary} and Theorem \ref{finalcongruences}).\\  

\begin{thmb}\label{b}
 Let $N$ be coprime with $c$. \begin{enumerate} 
\item Let $\pi_p: Y_{\operatorname{GU}(2,2)}(V_{Np,M}) \longrightarrow  Y_{\operatorname{GU}(2,2)}(V_{N,M})$ be the natural degeneracy map, then 
\[ \operatorname{Tr}_{\pi_p}({}_c\mathcal{Z}_{Np,M,u})=\begin{cases}{}_c\mathcal{Z}_{N,M,u}  & \text{ if }p\mid N; \\ 
\left( id - (\begin{smallmatrix} r I & 0 \\ 0 & r I \end{smallmatrix})^* \right){}_c\mathcal{Z}_{N,M,u} & \text{ if }p \nmid N,M  \end{cases}\]
 where $r$ denotes the inverse of $N$ in $(\Z/p\Z)^*$.
\item Let $p\nmid N,M$ and let $n,m$ be integers such that $m\geq 1$ and $n\geq 3m+3$. Let $\operatorname{Tr}_{f_p}$ be the trace map of the degeneracy map \[f_p:Y_{\operatorname{GU}(2,2)}(V_{Np^n,Mp^{m+1}}) \longrightarrow Y_{\operatorname{GU}(2,2)}(V_{Np^n,Mp^{m}})\] induced by right multiplication by $\operatorname{diag}(p^{-3},p^{-2},p^{-1},1)\in \operatorname{GU}(2,2)(\Q_p)$. There exists $u \in \operatorname{GU}(2,2)(\Af)$ such that 
\[ \operatorname{Tr}_{f_p}({}_c\mathcal{Z}_{Np^n,Mp^{m+1},u})=\mathcal{U}_p\cdot {}_c\mathcal{Z}_{Np^n,Mp^m,u},\]
where $\mathcal{U}_p$ is the Hecke operator of $\operatorname{diag}(p^{-3},p^{-2},p^{-1},1) \in \operatorname{GU}(2,2)(\Q_p)$.
\end{enumerate}
\end{thmb} 
In Theorem \ref{b}$(2)$, $p$ is either inert or split in the imaginary quadratic field which appears in the definition of $\operatorname{GU(2,2)}$. While Theorem \ref{b}(1) immediately follows from Theorem \ref{a}, the proof of Theorem \ref{b}(2) is more elaborate and it is based on ideas which have been employed in the proof of \cite[Theorem 5.4.1]{KLZ} for the \emph{vertical} Euler system norm relation of the Beilinson-Flach classes.\\ 

Along the way, we construct $\Lambda$-adic Eisenstein classes from Faltings' classes (Definition \ref{definitionlambdaadicclass}), and in Proposition \ref{comparisonfk} and Theorem \ref{secondcomparisonfk}, we compare them to the Eisenstein-Iwasawa classes introduced in \cite{padicinterpolation}.

\subsection{Structure of the paper.} In \S \ref{prelim}, we carry out the necessary preliminaries to show the push-forward relations of Faltings' classes; this involves studying in detail the tower of integral level structures corresponding to $\{ U_{N} \}_N$. In \S \ref{eisclassabschm}, we recall the definition of Eisenstein classes in both the elliptic and higher dimension case, and we discuss their basic properties, which should be well-known. We then use results from the previous section to deduce the compatibility of Faltings' classes in the mira-Klingen tower (\S \ref{miraklingen}). In \S \ref{iwtheisclass}, we discuss the relation between Faltings' classes and Kings' Eisenstein-Iwasawa classes.  Finally, after studying the embedding $\phi$ of $Y_{\GSp_4}$ into $Y_{\operatorname{GU}(2,2)}$  in \S \ref{genongu22shimvar}, we establish the desired compatibility relations of the push-forward of Faltings' classes along $\phi$ in \S \ref{twovarfamily}. 

\subsection{Acknowledgements.} First and foremost I would like to thank my supervisor Sarah Zerbes for suggesting the problem and for her continued support and guidance. I am also indebted to David Loeffler who, in several occasions, generously gave of his time to explain various aspects of the subject to me. I thank Chris Skinner and Guido Kings for numerous helpful discussions. Special thanks to Gregorio Baldi and Joaqu\'in Rodrigues Jacinto for their comments on the draft.  
Parts of this paper were written while I was visiting the Bernoulli Center at EPFL. I am very grateful to the Bernoulli Center and specially to Dimitar Jetchev for their hospitality and invitation to participate in part of the semester "Euler Systems and Special Values of L-functions".
 \section{Preliminaries}\label{prelim}
\subsection{Symplectic groups}

Denote by $I_g'$ the $g \times g$ anti-diagonal matrix with all entries 1 and $J=\left( \begin{smallmatrix}  & I_g' \\ -I'_g & \end{smallmatrix} \right)$. Let $\GSp_{2g}$ be the \[\GSp_{2g}(R) = \{ (h,m_h) \in (\GL_{2g} \times \GL_1)(R) : h^t J h=m_h J \}. \] Define the symplectic multiplier to be the homomorphism $$\mu : \GSp_{2g} \longrightarrow  \GL_1, \quad h \mapsto m_h.$$ It has kernel the symplectic group $\operatorname{Sp}_{2g}$.
 
\subsection{Symplectic Shimura varieties}

Let $\mathbb{S} = \mathrm{Res}_{\C / \R} \mathbb{G}_m$ and define $h \colon \mathbb{S} \to {\GSp_{2g}}_{/ \R}$ as follows. 
Let  $$ X := \left \{ M \in Sp_{2g}(\R) \quad  :  \quad
M^2=-I, \; \; 
\langle u , Mv \rangle := u^tJMv \text{ is} \pm\text{-definite } 
\right \} $$
be the set of positive or negative definite symplectic complex structures on the real vector space given by the standard representation of ${\GSp_{2g}}_{/ \R}$. The set $X$ can be identified with the set of homomorphisms \[h: \mathbb{S}  \longrightarrow {\GSp_{2g}}_{/\R},\] by sending $M \in X$ to $h$ such that $h(a+ib)=aI+bM.$ Every $\GSp_{2g}(\R)$-conjugacy class in $X$ defines a Shimura datum and, since two symplectic complex structures are $\GSp_{2g}(\R)$-conjugate, we conclude that $X$ consists of a single $\GSp_{2g}(\R)$-conjugacy class. 
The pair $(\GSp_{2g}, X)$ defines a Shimura datum with reflex field $\Q$. The attached Shimura variety is a moduli space of polarised abelian schemes with extra level structures, as we discuss in Subsection 1.3. Note that this Shimura datum naturally arises from a PEL datum of type C (e.g. see \cite{milneshim}, Definition 8.15 and Example 8.6).
\subsection{Moduli of abelian schemes for $\GSp_{2g}$ } 
We now recall the description of the moduli space interpretation of the Shimura varieties that will be used to study the trace compatibilities of the Eisenstein classes.

For any open compact subgroup $U$ of $\GSp_{2g}(\widehat{\Z})$, we can associate the set-valued functor $F_U$ from the category $Sch/\Q$ of schemes over $\Q$, which parametrises (isomorphism classes of) abelian schemes of relative dimension g with principal polarisation and $U$-level structure (whose definition is discussed later). Recall that if $U$ is neat (in the sense of \cite{pink}, p. 13), then $F_U$ is known to be representable by a smooth quasi-projective scheme $Y_{\GSp_{2g}}(U)$ over $\Q$ (for instance, see \cite{cptpel} Theorem 1.4.1.11). 
\begin{remark} Examples of neat subgroups are given by principal level subgroups \[K(N):= \{ h \in \GSp_{2g}(\widehat{\Z}): h \equiv I \mbox{ mod }N \}\] for an integer $N \geq 3$, and the open compact subgroups contained in it. We would like to remark that neatness ensures that the the only automorphism of each object in our moduli space is the identity. 
\end{remark}

\subsubsection{Integral symplectic level structures}
In the following, we review the definition of level structure for any open compact subgroup $U$ of $\GSp_{2g}(\widehat{\Z})$. We closely follow \cite{cptpel}.  We will denote these structures as symplectic level structures simply to remark that we are dealing with the symplectic group $\GSp_{2g}$. 
Consider an abelian scheme $A$ of relative dimension $g$ over a locally Noetherian $\Q$-scheme $S$ and fix a principal polarisation $\lambda$ on it; a "naive" candidate for the definition of a full level $N$-structure is the following.
\begin{definition}\label{def1} A naive symplectic full level $N$ structure on $(A,\lambda)_{/S}$ is an isomorphism $$\alpha_N:(\Z/N\Z)^{2g}_{/_S} \longrightarrow A[N],$$ which respects the symplectic forms defined by $J$ on $(\Z/N\Z)^{2g}$ and the one induced by the Weil pairing $e_{\lambda}$ and $\lambda$ on $A[N]$. 
\end{definition}

\begin{remark}
  The isomorphism $\alpha_N$ respects the two symplectic forms in the following sense. There exists an isomorphism $\beta_N:((\Z/N\Z)(1))_{/_S} \longrightarrow \mu_{N/_S}$ which makes the diagram 
  $$\xymatrix{(\Z/N\Z)^{2g}_{/_S}{\times}_S (\Z/N\Z)^{2g}_{/_S}\ar[rr]^{J} \ar[d]_{\alpha_N \times \alpha_N} & & ((\Z/N\Z)(1))_{/_S} \ar[d]^{\beta_N} \\ 
  A[N] {\times}_S A[N] \ar[rr]^{e_{\lambda}} & & \mu_{N/_S}
    }$$
    commutative.
 \end{remark}
For each geometric point $\bar{s}$ of $S$, define the Tate module (at $\bar{s}$) of $A$ to be the $\widehat{\Z}$-module $$T_{\bar{s}}(A) := \varprojlim_{N}A[N](\bar{s}).$$
Since $S$ is a $\Q$-scheme, $T_{\bar{s}}(A)$ is a free $\widehat{\Z}$-module of rank $2g$. We will see that there is another way to define symplectic full level $N$ structures, at the level of Tate modules (after passing to geometric fibres of $A/S$), which is equivalent to the one of Definition \ref{def1}, due to the following. 
\begin{lemma}\label{actionpi} Let $S$ be a connected locally Noetherian scheme and fix a geometric point $\bar{s}\to S$; there is an equivalence between the category of locally constant constructible \'etale abelian sheaves over $S$ and the one of finite continuous $\pi_1(S,\bar{s})$-modules, given by sending $G$ to its geometric fibre $G_{\bar{s}}$.
\end{lemma}
In particular, since $A[N]$ is a locally constant constructible \'etale sheaf on $S$, $A[N](\bar{s})$ has the structure of a $\pi_1(S,\bar{s})$-module. The $\widehat{\Z}$-module $T_{\bar{s}}(A)$ acquires an action of $\pi_1(S,\bar{s})$ from the one of each group of $N$-torsion points.
In view of Lemma \ref{actionpi}, we can translate the information given by a naive symplectic full level $N$ structure $\alpha_N$ in terms of a $\pi_1(S,\bar{s})$-equivariant isomorphism at the geometric fibre, say ${\alpha_N}_{\bar{s}}$. Note that ${\alpha_N}_{\bar{s}}$ is the reduction modulo $K(N)$ of a symplectic isomorphism $$\alpha_{\bar{s}}:\widehat{\Z}^{2g}\longrightarrow T_{\bar{s}}(A),$$
and such a lift is unique up to action of $K(N)$. We assume that the element $h \in \GSp_{2g}(\widehat{\Z})$ acts on the isomorphism ${\alpha}_{\bar{s}}$ by $\alpha_{\bar{s}} \circ h$, while $\sigma \in \pi_1(S,\bar{s})$ acts on the left. Hence, we can give the following definition.
 \begin{definition}\label{def2} A symplectic full level $N$-structure on $(A,\lambda)_{/_S}$ (at $\bar{s}$) is a $\pi_1(S,\bar{s})$-invariant $K(N)$-orbit of a symplectic isomorphism $$\alpha_{\bar{s}}:\widehat{\Z}^{2g}\longrightarrow T_{\bar{s}}(A).$$ Thus, a symplectic level $N$-structure on $(A,\lambda)_{/_S}$ is a collection of symplectic full level $N$-structures at each geometric point $$\{\alpha_{\bar{s}}\}_{\bar{s}},$$ such that if two geometric points $\bar{s},\bar{r}$ are in the same connected component, then  \[ \alpha_{\bar{s}}=\alpha_{\bar{r}}.\]
\end{definition} 
 In Definition \ref{def2}, the $\pi_1(S,\bar{s})$-invariance of the $K(N)$-orbit of $\alpha_{\bar{s}}$ is equivalent to asking the symplectic full level $N$ structure in the sense of Definition \ref{def1} to be defined over $S$ and hence it is an essential ingredient to compare the two definitions, as the next proposition shows.
\begin{proposition}[\cite{cptpel} 1.3.6.5-1.3.6.6]
Let $(A,\lambda)_{/_S}$ be as above; a symplectic level $N$-structure on $(A,\lambda)_{/_S}$ is equivalent to a tower $$(t_M:S_M \longrightarrow S)_{N|M}$$ of finite \'etale surjective maps such that: \begin{enumerate}
\item $S_N=S$ and for any $N|M|L$ there are finite \'etale surjective maps $$g_{L,M}:S_L \longrightarrow S_M $$ such that $t_L=t_M \circ g_{L,M}$;
\item Over each $S_M$, we have a symplectic isomorphism $$\alpha_{/_{S_M}}:(\Z/M\Z)^{2g}_{/_{S_M}} \longrightarrow A[M]_{/_{S_M}},$$ such that, if $N|L|M$, the pullback of $\alpha_{/_{S_L}}$ under $g_{M,L}$ is the reduction modulo $L$ of $\alpha_{/_{S_M}}$.
\end{enumerate}

\end{proposition}

The proposition above suggests an alternative (and more convenient to us) way to define level structures for general open compact subgroups of $\GSp_{2g}(\widehat{\Z})$.
\begin{definition}
Let $U$ be an open compact subgroup of $\GSp_{2g}(\widehat{\Z})$ and for any integer $M$ such that $K(M) \subset U$ denote by $U_M$ the quotient $U/K(M)$. Then, a symplectic level $U$-structure of $(A,\lambda)_{/S}$ is a collection $\{ \alpha_{U_M} \}_{M}$, where $M$ varies among the integers such that $K(M) \subset U$, of elements $\alpha_{U_M}$ such that \begin{enumerate} 
\item $\alpha_{U_M}$ is a locally \'etale defined $U_M$-orbit of a naive symplectic full level $M$-structure;
\item If $L|M$, $\alpha_{U_L}$ corresponds to the reduction modulo $L$ of $\alpha_{U_M}$.
\end{enumerate} 
\end{definition}


We finally note that passing to geometric fibres, a symplectic level $U$-structure gives a $\pi_1(S,\bar{s})$-invariant $U$-orbit of a symplectic isomorphism $$\alpha:\widehat{\Z}^{2g}\longrightarrow T_{\bar{s}}(A),$$
at each geometric point.
If $U$ is neat, then this allows us to describe the complex points of the quasi-projective $\Q$-scheme $Y_{\GSp_{2g}}(U)$ as the ones of the Shimura variety attached to $(\GSp_{2g},X)$ \[ Y_{\GSp_{2g}}(U)(\C) =\GSp_{2g}(\Q)\backslash X\times \GSp_{2g}(\A_f)/ U. \]

\subsubsection{Tower of symplectic level structures at $p$}

One of the main results of this article involves the computation of distribution relations for \'etale Eisenstein classes attached to moduli of abelian schemes with specific level structures at a prime $p$.
 For representability issues, we work with open compact subgroups $U$ which decompose as $U_p\cdot U^{(p)} \subset \GSp_{2g}(\A_{f})$, where $U_p  \subset \GSp_{2g}(\Q_p)$ and we fix a neat $U^{(p)} \subset \GSp_{2g}(\widehat{\Z}^{(p)}).$
Under this assumption, $U=U^{(p)}\GSp_{2g}(\Z_p)$ is neat and there is a scheme $Y_{\GSp_{2g}}(U)/ \Q$ that parametrises p.p. abelian schemes (of rel. dim. $g$) with $U$-level structure. Let $\Am=\Am_g(U)$ denote its universal abelian scheme and consider the following functor 
$$G_1(p^m):Sch_{/_{Y_{\GSp_{2g}}(U)}}  \longrightarrow Sets,$$ defined by $$S/Y_{\GSp_{2g}}(U) \mapsto \{ \mbox{points of exact order } p^m \mbox{ of } \Am \times_{Y_{\GSp_{2g}}(U)} S/S \}.$$

\begin{remark}
Since we are working in characteristic zero, by point of exact order $p^m$ of $A/S$, we simply mean a section $S \to A$ whose pull-back to each geometric fibre is a point of exact order $p^m$.
\end{remark}

We now compare $G_1(p^m)$ with the sheaf induced by the following open compact subgroups of $\GSp_{2g}(\Z_p)$.
\begin{definition}\label{symplecticklingen}  For any integer $m \geq 1$, define the subgroup $U_1(p^m) \subset \GSp_{2g}(\Z_p)$ as follows: \begin{eqnarray}
U_1(p^m)&:=  &\{ M \in  \GSp_{2g}(\Z_p) | R_{2g}(M) \equiv (0, \cdots, 0,1) \text{ mod }p^m \} 
\end{eqnarray}
where $R_i(M)$ denotes the $i$-th row of $M$.
If $N= \prod p_i^{e_i}$, then $U_1(N) \subset \GSp_{2g}(\widehat{\Z})$ is defined to be the subgroup of elements $(g_p)_p$ such that $g_{p_i}\in U_1(p_i^{e_i})$.
\end{definition}

\begin{lemma}\label{pointsandls} Let $A/S$ be an abelian scheme of relative dimension g over a $\Q$-scheme $S$, with a fixed principal polarisation on it. Then, there is a bijection between points of exact order $p^m$ of $A$ and symplectic level $U_1(p^m)$-structures.  
\end{lemma} 
\begin{proof}
Denote by $U_{p^m}$ the image of $U_1(p^m)$ under reduction mod $p^m$. 
Since $A$ is of finite presentation over $S$, we can reduce to work over a locally Noetherian base $S$; moreover, since $S$ is the disjoint union of its connected components and \'etale sheaves send co-products into products, it is sufficient to work over a connected locally Noetherian $S$.
 Finally, by replacing $S$ by an \'etale finite surjective cover of it if necessary, we can  assume $$A[p^m] \simeq (\Z/p^m\Z)^{2g}_{/S}.$$ A point $t \in A(S)$ of exact order $p^m$  lifts to a symplectic "naive" level $p^m$-structure and such a lift is unique up to action of $U_{p^m}$. Indeed, $t$ defines a monomorphism over $S$ $$t:({\Z/p^m\Z})_{/S} \hookrightarrow A[p^m],$$ and it can be completed to a full isomorphism $$(\Z/p^m\Z)^{2g}_{/S}\longrightarrow A[p^m]$$ uniquely up to the action of $U_{p^m}$.

 After passing to a suitable \'etale cover of $S$, we can lift $t$ to a point of exact order $p^{m+1}$ of $A$, which is mapped to $t$ under (the abstract group homomorphism) reduction modulo $p^m$. Repeating this procedure for any $l \geq m$ uniquely defines a level $U_1(p^m)$-structure of $A$. The converse is proved similarly.
\end{proof}

This result directly implies the following.
\begin{cor}\label{corollaryonlevelsatp} The scheme $\mathcal{G}_1(p^m)$ which represents $G_1(p^m)$ is isomorphic to $Y_{\GSp_{2g}}(U^{(p)}U_1(p^m))$ as a covering of $Y_{\GSp_{2g}}(U^{(p)}\GSp_{2g}(\Z_p))$.
\end{cor} 
This generalises a well-known result for modular curves (e.g. \cite[Theorem 4.3.3]{KLZ}),  which plays an important role in the study of the pushforward relations of Eisenstein classes for $\GL_2$. As in \emph{loc. cit.}, we use Corollary \ref{corollaryonlevelsatp} to prove pushforward relations of \'etale Eisenstein classes for $\GSp_{2g}$ in Subsection 2.4. 
\begin{remark}As a consequence of the Chinese Remainder Theorem, if $$N=\prod_{i=1}^r p_i^{e_i},$$ $Y_{\GSp_{2g}}(U^{(N)}U_1(N))$ parametrises p.p. abelian schemes of relative dimension $g$ with level structure $U^{(N)}$ and $r$ different points each of exact order $p_i^{e_i}$. 
\end{remark}

\subsubsection{Integral models of symplectic Shimura varieties}
In the following, we recall the existence of integral models for the symplectic Shimura variety $Y_{\GSp_{2g}}$ of level $U^{(p)}U_1(p^r)$, which is needed in Subsection 2.6. We refer to \cite[Section 3]{moonen} or \cite[Section 6.4.1]{hida} for further details. By \cite[Theorem 7.10]{git}, there exists a (fine) moduli space over $\Z[\frac{1}{dp}]$, for an auxiliary integer $d \geq 3$ coprime to $p$ which depends on $U^{(p)}$, of isomorphism classes of principally polarised abelian schemes with symplectic level $U(p^r):=U^{(p)}K(p^r)$-structure. We denote it by $\mathcal{Y}_{\GSp_{2g}}(U(p^r))_{/\Z[\frac{1}{dp}]}$. Let $\mathcal{Y}_{\GSp_{2g}}(U^{(p)}U_1(p^r))_{/\Z[\frac{1}{dp}]}$ be its quotient by $U_1(p^r)/K(p^r)$. 

Similarly to the previous subsection, let $\mathcal{G}_1(p^r)$ be the finite \'etale sheaf over $\mathcal{Y}_{\GSp_{2g}}(U^{(p)}\GSp_{2g}(\Z_p))_{/\Z[\frac{1}{dp}]}$ associated to points of exact order $p^r$ of the universal abelian scheme of $\mathcal{Y}_{\GSp_{2g}}(U^{(p)}\GSp_{2g}(\Z_p))_{/\Z[\frac{1}{dp}]}$.
Then, Corollary \ref{corollaryonlevelsatp} is still true in this setting. 
\begin{lemma}\label{integralmodelversion}
The scheme $\mathcal{G}_1(p^r)$ is isomorphic to $\mathcal{Y}_{\GSp_{2g}}(U^{(p)}U_1(p^r))$ as a covering of $\mathcal{Y}_{\GSp_{2g}}(U^{(p)}\GSp_{2g}(\Z_p))$.
\end{lemma}
\begin{proof} Since $p$ is invertible in $\mathcal{Y}_{\GSp_{2g}}(U^{(p)}\GSp_{2g}(\Z_p))_{/\Z[\frac{1}{dp}]}$, the proof is identical to the one of Corollary \ref{corollaryonlevelsatp}.
\end{proof}
\subsection{Continuous \'etale cohomology}\label{continuousetalecohomology}
 In this article, we work with continuous \'etale cohomology for schemes over a general base, introduced by Jannsen in \cite{Jannsen}. We now recall its definition.
\begin{definition}\label{defetcontcohom} For an inverse system $(\mathcal{F}_n)$ of constructible \'etale $\Z/p^n$-sheaves over a scheme $X$, define $H^i_\et(X,(\mathcal{F}_n))$ to be the $i$-th derived functor of $(\mathcal{F}_n) \mapsto \varprojlim_n  H^0_{\et}(X,\Ff_n)$. In particular, for $p$ invertible on $X$ and an integer $j$, we define \[ H^i_\et(X,\Z_{p}(j)):=H^i_\et(X,(\Z/p^n\Z(j))). \] 
\end{definition}
Note that if $H^{i-1}_{\et}(X,\Ff_n)$ is finite for all $n$, then \[H^i_\et(X,(\Ff_n)) \simeq \varprojlim_n H^{i}_{\et}(X,\Ff_n).\] 

\begin{remark}\label{nicepropofcohom}
For instance, this last condition is satisfied whenever $X$ is a scheme over $S$, where the base $S$ is an algebraically closed field or it is a scheme of finite type over $\Z$.
\end{remark}
Finally, for $\Ff=(\Ff_n)_n$ as in Definition \ref{defetcontcohom}, we denote $\Ff \otimes_{\Z_p} \Q_p$ by $\Ff_{\Q_p}$ and  define 
\[ H^i_\et(X,\Ff_{\Q_p}) :=H^i_\et(X,\Ff) \otimes_{\Z_p} \Q_p.\]

\section{Eisenstein classes for abelian schemes}\label{eisclassabschm}
\subsection{Recap of the case of elliptic curves}
Let $\pi:E \to S$ be an elliptic curve over the scheme $S$ and fix a prime $p$ invertible on $S$. For any integer $c$, consider \begin{align*}
U_c:= & E \smallsetminus E[c]; \\
\pi_c:= & \pi|_{U_c}.
\end{align*}
In \cite{kato}, Kato defines Siegel units as the evaluation at torsion points of certain canonical functions in $\O(U_c)^*$. These units are constructed from Cartier divisors, which are invariant under norm maps. Indeed, recall we have the following.
\begin{theorem}[\cite{kato}, Proposition 1.3]\label{theo1}Let $E$ be an elliptic curve over a scheme $S$ and fix an integer $c$ prime to 6. Then, there exists a unique element $_c\theta_E \in \O(U_c)^*$ such that: 
\begin{enumerate}
\item $_c\theta_E$ has divisor $c^2(0)-E[c]$ on $E$, where the zero-section $(0)$ of $E$ and the kernel of the multiplication-by-$c$ morphism $E[c]$ are regarded as Cartier divisors;
\item for any integer $a$ coprime to $c$, $_c\theta_E$ is  compatible under the norm map $N_a:\O(U_{ac})^* \to \O(U_c)^*$ associated to the pullback of the multiplication-by-$a$ map $U_{ac} \to U_c$, i.e. $$N_a({_c\theta_E})= {_c\theta_E};$$  
\item for any integers $c_1,c_2$ prime to $6$, we have
\[_{c_1 \cdot c_2}\theta_E=c_1^*({_{c_2}\theta_E}) \cdot {_{c_1}\theta_E}^{c_2^2}.\]
\end{enumerate}
\end{theorem}
Kato defines Siegel units as pullback by torsion sections of the units $_c\theta_{E}$, associated to  the universal elliptic curve $E/Y(N)$. 
\begin{definition}
Fix an integer $N$ and an integer $c$ coprime with $6N$.
Let $(E,e_1,e_2)$ be the universal elliptic curve over the modular curve $Y(N)$ with full level $N$ structure, then for any $N$-torsion section $x=ae_1+be_2$, with $a,b \in \Z/N\Z$, we define the Siegel unit $$_cg_x := x^*{_c\theta_E} \in \O(Y(N))^*.$$ 
\end{definition}
The idea behind Theorem \ref{theo1} relies on the fact that a Cartier divisor, which by definition is locally principal, that satisfies the rigid condition given by Theorem \ref{theo1}$(2)$ is globally principal. Once wework with an abelian scheme $A$ of dimension $g$ higher than 1, this strategy does not apply since $c^{2g}(0) - A[c]$ is not a Cartier divisor. We can state Theorem \ref{theo1} in a more convenient way for our purposes.

\begin{definition}\label {residuemap}
We define the motivic residue map for $E/S$ to be the map 
$$res:H^1_{\text{mot}}(U_c,\Z(1)) \longrightarrow H^0_{\text{mot}}(E[c],\Z)^{deg=0},$$ 
 which comes from the long exact (Gysin) sequence associated to the triple $U_c\hookrightarrow E \hookleftarrow E[c]$.
 \end{definition}
Recall that $H^1_{\rm mot}(U_c,\Z(1))=\O(U_c)^*$ and that $H^0_{\rm mot}(E[c],\Z)^{deg=0}$ is identified with the group of degree 0 divisors supported in $E[c]$. 
We can re-interpret Theorem \ref{theo1} as follows.
\begin{theorem}\label{cohosiegel}
Let $E$ be an elliptic curve over $S$ and let $c$ be an integer prime to 6.
The divisor $c^2(0)-E[c]$ lifts canonically under the residue map $$res:H^1_{mot}(U_c,\Z(1))\longrightarrow H^0_{mot}(E[c],\Z)^{deg=0}$$ to an element $_c\theta_E \in H^1_{mot}(U_c,\Z(1))$ such that, for any integer $a$ coprime to $c$, $$N_a({_c\theta_E})={_c\theta_E},$$ and,  if $c_1,c_2$ are integers prime to $6$, we have
\[_{c_1 \cdot c_2}\theta_E=c_1^*({_{c_2}\theta_E}) \cdot {_{c_1}\theta_E}^{c_2^2}.\]
\end{theorem}
This cohomological interpretation extends to other cohomology theories such as the \'etale one.

\subsubsection{The \'etale classes for elliptic curves}
In the following, we recall how to construct \'etale cohomology classes in the elliptic curve case. As a direct consequence of Theorem \ref{theo1}, we have classes $${_c\theta^{\et}_E} \in H_{\et}^{1}(U_c,\Z_{p}(1)),$$ which satisfy the following: \begin{itemize}
\item[P1.] For any $r$ prime to $c$, ${_c\theta^{\et}_E}$ is invariant under trace maps associated to multiplication by $r$, i.e. $$\operatorname{Tr}_r({_c\theta^{\et}_E})={_c\theta^{\et}_E};$$
\item[P2.] The classes ${_{c_1\cdot c_2}\theta^{\et}_E}$, ${_{c_1}\theta^{\et}_E}$ and $_{c_2}\theta^{\et}_E$ are related by \[_{c_1\cdot c_2}\theta^{\et}_E=[c_1]^*({_{c_2}\theta^{\et}_E})+c_2^{2}{_{c_1}\theta^{\et}_E}=[c_2]^*({_{c_1}\theta^{\et}_E})+c_1^{2}{_{c_2}\theta^{\et}_E}.\]
\end{itemize}

\begin{remark}
Recall that the trace map associated to multiplication by $r$ prime to $c$ is the composition of \[\operatorname{Tr}_r:H_{\et}^{1}(U_c,\Z_{p}(1)) \longrightarrow H_{\et}^{1}(U_{rc},\Z_{p}(1)) \longrightarrow H_{\et}^{1}(U_c,\Z_{p}(1)),\] where the first map is just restriction to $U_{rc}$ and the second is the trace map associated to $[r]$.
\end{remark}
The \'etale cohomology classes associated to elliptic curves are the \'etale realisation of the units of Theorem \ref{theo1}. 
Indeed, using the connecting homomorphism $$\partial_{p^t}:\O(U_c)^*\to H^1_{\et}(U_c,\Z/p^t\Z(1))$$ of the Kummer exact sequence of \'etale sheaves \begin{displaymath}
\xymatrix{
0 \ar[r] & \mu_{p^t} \ar[r] & \mathbb{G}_m \ar[r]^{(\cdot)^{p^t}} &\mathbb{G}_m \ar[r] & 0
}
\end{displaymath}
we get the element $\partial_{p^t}({_c\theta_E})\in H^1_{\et}(U_c,\Z/p^t\Z(1))$.

\begin{prop}\label{prop2.7}
The class ${_c\theta^{\et}_E}:=\varprojlim_t \partial_{p^t}({_c\theta_E})\in H^1_{\et}(U_c,\Z_{p}(1))$ satisfies the properties $P1,P2$ listed above.
\end{prop}
\begin{proof}
It is enough to show the statement at finite levels.
The units ${_c\theta_E}$ are invariant under norm maps $N_r$ associated to multiplication by $r$, for $r$ coprime with $c$. Hence, property $P1$ follows from the commutative diagram \begin{displaymath}
\xymatrix{
\O(U_c)^*\ar[r]^-{\partial_{p_t}} \ar[d]_{\tilde{N}_r} &  H^1_{\et}(U_c,\Z/p^t\Z(1)) \ar[d]^{\operatorname{Tr}_r} \\ 
\O(U_c)^*\ar[r]^-{\partial_{p_t}} &  H^1_{\et}(U_c,\Z/p^t\Z(1)), 
}
\end{displaymath}
where $\tilde{N}_r$ is the obtained by composing $N_r$ on the right by the restriction map $\O(U_c)^*\to \O(U_{rc})^*$.
Using the same diagram, property $(ii)$ follows from the fact that 
\begin{align}\label{siegelformuladist}
_{c_1 \cdot c_2}\theta_E=c_1^*({_{c_2}\theta_E}) \cdot {_{c_1}\theta_E}^{c_2^2}.
\end{align}
Formula \eqref{siegelformuladist} is obtained in \cite{kato} by noticing that both sides have the same divisor and that their ratio must be in $\O(S)^*$, hence it is equal to 1, by the norm compatibility property of the units $_c\theta_E$.
\end{proof} 
The \'etale Eisenstein classes are obtained as pull-back under torsion sections of these \'etale cohomology classes associated to the universal elliptic curve of $Y(N)$.  
\subsubsection{The \'etale Gysin sequence}
Before passing to the higher dimensional case, we explicitly calculate the \'etale residue of the class ${_c\theta^{\et}_E}$ in the case where $E/K$ is an elliptic curve over an algebraically closed field $K$; in particular, we characterise ${_c\theta^{\et}_E}$ as the only class which satisfies the property $P1$ described above, and, with \'etale residue $c^2(0)-E[c]$. Note that Theorem \ref{theo1} allows us to canonically choose a class in $H^1_{\et}(U_c,\Z_{p}(1))$ with the desired properties; however, it is possible to lift canonically a suitable multiple of $c^{2}(0)-E[c]$, as we will see in the next section, by invoking a theorem of Faltings (see Theorem \ref{thetheorem}).\newline 
Consider the \'etale Gysin exact sequence for $$E[c] \hookrightarrow E \hookleftarrow U_c.$$ It gives
\begin{displaymath}
\xymatrix{
0\ar[r] & 0 \ar[r] &
H^0_{\et}(E,\Z/p^t\Z(1))\ar[r]^-{\backsim} &
H^0_{\et}(U_c,\Z/p^t\Z(1)) \ar@{->} `r/8pt[d] `/15pt[l] `^dl[ll]|{} `^r/1pt[dll] [dll] \\  
 & 0 \ar[r] &
H^1_{\et}(E,\Z/p^t\Z(1)) \ar@{^{(}->}[r]
 &
H^1_{\et}(U_c,\Z/p^t\Z(1)) \ar@{->} `r/5pt[d] `/15pt[l] `^dl[lll]|{} `^r/2pt[dll] [dll] \\
 & H^0_{\et}(E[c],\Z/p^t\Z) \ar@{->>}[r] &
H^2_{\et}(E,\Z/p^t\Z(1)) \ar[r] & 0
}
\end{displaymath}
 Let $res_{\et}$ denote the map of the \'etale Gysin sequence (with $\Z_{p}$-coefficients) \[H^1_{\et}(U_c,\Z_{p}(1)) \longrightarrow H^0_{\et}(E[c],\Z_{p}).\]  
\begin{lemma}\label{ressiegel} The \'etale cohomology class ${_c\theta^{\et}_E}$ is the unique class fixed by $\operatorname{Tr}_r$, for $(r,c)=1$, such that \[ res_{\et}({_c\theta^{\et}_E})=c^2(0)-E[c]. \]
\end{lemma}
\begin{proof}
The \'etale residue map can be explicitly described as follows.
Note that the quotient \[Q_t=H^1_{\et}(U_c,\Z/p^t\Z(1))/H^1_{\et}(E,\Z/p^t\Z(1))\]
can be seen as the group of degree 0 divisors with coefficients in $\Z/p^t\Z$ and supported on $E[c]$. Indeed, recall that $H^1_{\et}(E,\Z/p^t\Z(1))\simeq Pic^0(E)[p^t],$ while \begin{align*}
H^1_{\et}(U_c,\Z/p^t\Z(1)) &\simeq \left \{(\mathcal{L},f) | \mathcal{L} \in Pic(U_c), f:\mathcal{L}^{\otimes^{p^t}}\overset{\sim}{\longrightarrow} \O_{U_c} \right \}/\cong \\
&\simeq \left \{(\mathcal{L},D,f) | \mathcal{L} \in Pic^0(E), f:\mathcal{L}^{\otimes^{p^t}}\overset{\sim}{\longrightarrow} \O(D) \right \}/\langle (\O(D'),p^tD',1^{\otimes^{p^t}}) \rangle.
\end{align*}
where $D$ and $D'$ are degree 0 divisors supported on $E[c]$ (e.g. see \cite[\href{http://stacks.math.columbia.edu/tag/03RR}{Tag 03RR}]{stacksproj}). Hence, to each element of $Q_t$ we can associate a divisor with coefficients in $\Z/p^t\Z$ supported on the $c$-torsion points, and the \'etale residue map is described as \[res_{\et}:H^1_{\et}(U_c,\Z/p^t\Z(1)) \longrightarrow Q_t \hookrightarrow H^0_{\et}(E[c],\Z/p^t\Z), \quad (\mathcal{L},D,f) \mapsto D.\] Passing to the limit, we get a similar description for the quotient with $\Z_{p}$-coefficients. 
By Theorem \ref{theo1}, as an element of $Q_t$ the class $\partial_{p^t}({_c\theta_E})$ is the one associated to the degree 0 divisor $c^2(0)-E[c]$ (i.e. the associated class in the kernel of $H^0_{\et}(E[c],\Z/p^t\Z) \to H^2_{\et}(E,\Z/p^t\Z(1))$. Thus, we conclude that \[res_{\et}({_c\theta^{\et}_E})=c^2(0)-E[c].\]
To prove uniqueness, we use the fact that the Gysin sequence is equivariant for the action of $\operatorname{Tr}_r$. Indeed, suppose that there is $d_c \in H^1_{\et}(U_c,\Z_{p}(1))^{\operatorname{Tr}_r=1},$ for an integer $r$ prime to $c$, such that $res_{\et}(d_c)=c^2(0)-E[c]$; then, the difference ${_c\theta^{\et}_E}-d_c$ is fixed by $\operatorname{Tr}_r$ and lies in $H^1_{\et}(E,\Z_{p}(1))$. Since $\operatorname{Tr}_r$ acts as multiplication by $r$ on $H^1_{\et}(E,\Z_{p}(1))$, we conclude that \[ {_c\theta^{\et}_E}-d_c=0.\]
\end{proof}
\begin{remark}
The explicit description in Lemma \ref{ressiegel} of the \'etale residue map \[res_{\et}:H^1_{\et}(U_c,\Z_{p}(1)) \longrightarrow H^0_{\et}(E[c],\Z_{p}) \] builds a direct analogy with the motivic residue map of Definition \ref{residuemap}. Indeed, we have the commutative diagram \[ \xymatrix{ H^1_{mot}(U_c,\Z(1)) \ar[d]_{ \varprojlim_t \partial_{p^t}} \ar[rr]^{res} &  & H^0_{mot}(E[c],\Z)^{deg=0} \ar[d]_{r_{\et}} \\ H^1_{\et}(U_c,\Z_{p}(1)) \ar[rr]^{res_{\et}} & & H^0_{\et}(E[c],\Z_{p})^{deg=0},  } \] where $H^0_{\et}(E[c],\Z_{p})^{deg=0}:=\operatorname{Ker} \left( H^0_{\et}(E[c],\Z_{p}) \to H^2_{\et}(E,\Z_{p}(1))\right)$, and  $r_{\et}(D)$ is the \'etale characteristic class of $D$, simply defined by sending the divisor $D$ to itself (now seen as a divisor with coefficients in $\Z_{p}$). \end{remark}
\subsection{\'Etale cohomology classes in higher dimension}
In the following, we describe how Theorem \ref{theo1} generalises to the case of \'etale cohomology classes for abelian schemes of relative dimension $g$, following Faltings' construction in \cite[Section 3]{Faltings2005}. \\ 
 Let $\pi:A \longrightarrow S$ be an abelian scheme  of relative dimension $g$ and let $p$ be a prime invertible in $S$. For any integer $c \ne p$, let $U_c := A \smallsetminus A[c]$ with structure morphism $\pi_c: U_c \longrightarrow S$. Recall that the trace map associated to multiplication by $r$ prime to $c$ is the composition of \[\operatorname{Tr}_r:H_{\et}^{2g-1}(U_c,\Z_{p}(g)) \longrightarrow H_{\et}^{2g-1}(U_{rc},\Z_{p}(g)) \longrightarrow H_{\et}^{2g-1}(U_c,\Z_{p}(g)),\] where the first map is just restriction to $U_{rc}$ and the second is the trace map associated to $[r]$.
Then, in the case of abelian scheme $A/S$ of relative dimension $g$, Faltings constructs classes $$_cz^A_m \in H_{\et}^{2g-1}(U_c,\Z_{p}(g)),$$ which satisfy the direct modification of properties $P1,P2$, which were previously introduced in the case of elliptic curves: \begin{itemize}
\item[P1.] For any $r$ prime to $c$, $_cz^A_m$ is invariant under trace maps associated to multiplication by $r$, i.e. $$\operatorname{Tr}_r({_cz^A_m})={_cz^A_m};$$
\item[P2.] The classes ${_{c_1\cdot c_2}z}$, ${_{c_1}z^A_m}$ and $_{c_2}z^A_m$ are related by\[_{c_1\cdot c_2}z^A_m=[c_1]^*({_{c_2}z^A_m})+c_2^{2g}{_{c_1}z^A_m}=[c_2]^*({_{c_1}z^A_m})+c_1^{2g}{_{c_2}z^A_m}.\]
\end{itemize}
Furthermore, pulling them back under torsion sections $x \in A(S)$ of order prime to $c$, we obtain the classes  $$_cz_{m,x}^A \in  H_{\et}^{2g-1}(S,\Z_{p}(g)).$$ 
We first recall the construction of those classes and then describe their properties in the sections 2.3, 2.4 and 2.5.
Consider the \'etale Gysin sequence for $$A[c] \hookrightarrow A \hookleftarrow U_c.$$ 
It tells us how the direct image \'etale sheaves (on $S$) $R^i\pi_{c,*}(\Z/p^t\Z(g))$ are related to $R^i\pi_{*}(\Z/p^t\Z(g))$. Let $\mathcal{H}^j(A[c],\mathcal{F})$ denote the $j^{th}$ direct image sheaf of the sheaf $\mathcal{F}$ on $A[c]$.
\begin{prop}\label{gysin}
The sheaf $R^i\pi_{c,*}(\Z/p^t\Z(g))$ coincides with $R^i\pi_*(\Z/p^t\Z(g))$ for $i<2g-1$; moreover, we have the exact sequence  $$0 \longrightarrow R^{2g-1}\pi_{*}(\Z/p^t\Z(g))\longrightarrow R^{2g-1}\pi_{c,*}(\Z/p^t\Z(g)) \longrightarrow \mathcal{H}^0(A[c],\Z/p^t\Z) \longrightarrow \Z/p^t\Z.$$   
\end{prop}
\begin{proof}
The result can be checked at geometric stalks. Then, it follows from Corollary VI.5.3 and Remark VI.5.4,(b) of \cite{milne1980etale}.  
\end{proof}
\begin{remark}
Similarly, this holds for the sheaf $\Z_{p}$.
\end{remark}

Recall that the global sections $\mathcal{H}^0(A[c],\Z_{p})(S)$ are by the very definition identified with $\Z_{p}(A[c])$, since $\pi_{|_{A[c]}}:A[c]\to S$ is a finite \'etale map. We now define the characteristic class of $c^{2g}(0)-A[c]$.

\begin{definition}\label{dc} Let $e:S \to A[c]$ be the closed immersion defined by the unit section; it induces \[ e_*: \Z_{p}(S) \longrightarrow \Z_{p}(A[c])= \mathcal{H}^0(A[c],\Z_{p})(S). \] We define the characteristic class of $c^{2g}(0)-A[c]$ to be the global section 
\[ D_c:=c^{2g}e_*(1)-\pi_{|_{A[c]}}^*(1) \in  \mathcal{H}^0(A[c],\Z_{p})(S). \]
\end{definition}

From Proposition \ref{gysin}, we have that the image of 
\[R^{2g-1}\pi_{*}(\Z_{p}(g))\hookrightarrow R^{2g-1}\pi_{c,*}(\Z_{p}(g))\] is isomorphic to the kernel of \[ \phi:\mathcal{H}^0(A[c],\Z_{p})\longrightarrow \Z_{p}.\] The next lemma shows that the characteristic class $D_c$ defines a global section of the kernel of $\phi$, hence of the image of $R^{2g-1}\pi_{*}(\Z_{p}(g))\hookrightarrow R^{2g-1}\pi_{c,*}(\Z_{p}(g))$. 

\begin{lemma}
The global section $D_c \in \mathcal{H}^0(A[c],\Z_{p})(S)$ lies in the kernel of $\phi(S)$.
\end{lemma}
\begin{proof}
Since the kernel of $\phi$ is a sheaf, we can reduce to show that the restriction of $D_c$ to each geometric fibre is zero. Thus, suppose that $A$ is an abelian scheme over an algebraically closed field, then we want to check that $D_c$ lies in the kernel of the Gysin map $$\varphi:H^0_{\et}(A[c],\Z_{p})\longrightarrow H^{2g}_{\et}(A,\Z_{p}(g)).$$ Note that $H^0_{\et}(A[c],\Z_{p})$ is isomorphic to ${\Z_{p}}^{c^{2g}},$ while $H^{2g}_{\et}(A,\Z_{p}(g))$ is isomorphic to $\Z_{p}$. By Poincar\'e duality, the $\Z_{p}$-dual of $H^{2g}_{\et}(A,\Z_{p}(g))$ is $H^0_{\et}(A,\Z_{p})$ which is isomorphic to $\Z_{p}$ and the pairing is induced by the product structure of $\Z_{p}$. On the other hand, $H^0_{\et}(A[c],\Z_{p})$ is dual to itself with respect to the pairing induced by multiplication of vectors $$\langle \bullet,\bullet \rangle :\Z_{p}^{c^{2g}}=H^0_{\et}(A[c],\Z_{p}) \times H^0_{\et}(A[c],\Z_{p})=\Z_{p}^{c^{2g}} \longrightarrow \Z_{p}, \begin{pmatrix}
  a_{1} \\ a_{2} \\ \vdots \\ a_{c^{2g}} 
 \end{pmatrix}, \begin{pmatrix}
  b_{1} \\
  b_2 \\
  \vdots \\
  b_{c^{2g}} 
 \end{pmatrix}\mapsto a_1b_1+ \cdots a_{c^{2g}}b_{c^{2g}}.$$
Then, the dual map to $\varphi$ is $$\Z_{p}=H^0_{\et}(A,\Z_{p})\longrightarrow H^0_{\et}(A[c],\Z_{p})=\Z_{p}^{c^{2g}}, a\mapsto(a, a, \cdots,a ). $$
Hence, let $v=(a_{1},a_{2}, \cdots, a_{c^{2g}}) \in \Z_{p}^{c^{2g}}$ and $b \in \Z_{p}$, then $$\langle \varphi(v),b \rangle=\langle v,\begin{pmatrix}
  b, & b, & \cdots, & b 
 \end{pmatrix} \rangle=b \cdot \sum a_i,$$ which gives us the result claimed.
\end{proof}

 Unfortunately, we cannot immediately lift $D_c$ to $H^{2g-1}_{\et}(U_c,\Z_{p}(g))$. As we said before, in the case of $g=1$, $c^{2g}(0)-A[c]$ is an effective Cartier divisor and we can use this geometric information to lift the class and construct Siegel units. Hence, we have a canonical choice of a class (the image of $_c\theta_A$ under the Kummer map) in $H^1_{\et}(U_c,\Z_{p}(1))$ whose residue is $D_c$. Faltings proposes a "cohomological" method to overcome this obstacle.
\begin{lemma}\label{operatore} Let $m$ be an integer prime to $p$ and let $\operatorname{Tr}_m$ denote the trace map associated to multiplication by $m$ on $A$, then the product $$F_m:=\prod_{i=0}^{2g-1}(\operatorname{Tr}_m-m^{2g-i})$$ annihilates the truncated complex $\tau_{\leq 2g-1}(\R\pi_{*}(\Z_{p}(g)))$.
\end{lemma}
\begin{proof} This is a direct consequence of the action of $\operatorname{Tr}_m$ in cohomology, which acts on the $i$-th degree cohomology as multiplication by $m^{2g-i}$. Indeed, note that \[R^i\pi_{*}(\Z_{p}(g)) \simeq R^{2g-i}\pi_{*}(\Z_{p}) \simeq  \bigwedge^{2g-i} R^{1}\pi_{*}(\Z_{p}) \simeq  \bigwedge^{2g-i} \left(\mathcal{T}_{p}(A) \right)^{\vee} \] where $\mathcal{T}_{p}(A)$ denotes the $p$-adic Tate module of A and $(\bullet)^{\vee}$ denotes its $\Z_{p}$-dual.
This follows by base change to geometric points and a duality statement for abelian varieties (\cite{milneabvar} Theorem 12.1). Thus, the result follows since $\operatorname{Tr}_m$ acts as multiplication by $m$ on $\mathcal{T}_{p}(A)$. \end{proof}

We now have the following result.

\begin{theorem}[\cite{Faltings2005}, Section 3]\label{thetheorem} Let $m$ be an integer coprime to $c$ and a generator of $\Z_{p}^*$ and define $N_m := \prod_{i=0}^{2g-1}(1-m^{2g-i})$. 
The class of $N_m^2(D_c)$ lifts canonically to $H^{2g-1}_{\et}(U_c,\Z_{p}(g))$.
\end{theorem}
\begin{proof} Here we give a sketch of the proof given in \cite{Faltings2005}. We have a surjective map \[ H^{2g-1}_{\et}(U_c,\Z_{p}(g)) \longrightarrow \operatorname{Ker}\left(H^{0}_{\et}(A[c],\Z_{p})\longrightarrow H^{2g}_{\et}(A,\Z_{p}(g)) \right),\] which comes from the \'etale Gysin exact sequence for $(A,U_c,A[c])$; denote by $\psi$ the map \[H^{0}_{\et}(A[c],\Z_{p})\longrightarrow H^{2g}_{\et}(A,\Z_{p}(g)).\]
First, we lift $D_c\in \operatorname{Ker}(\phi(S))$ to the kernel of $\psi$ to construct the desired class.  
To do so, Faltings constructs a map of complexes \[ \phi_{2g}:R^{2g}\pi_*\Z_{p}(g)[-2g] \longrightarrow \R\pi_*\Z_{p}(g) \] which is multiplication by $(2g)!$ in cohomology. Since, by Lemma \ref{operatore}, the operator \[F_m:=\prod_{i=0}^{2g-1}(\operatorname{Tr}_m-m^{2g-i})\] annihilates the truncated complex $\tau_{\leq 2g-1}(\R\pi_{*}(\Z_{p}(g)))$, then \[F_m \circ  \phi_{2g}=F_m \circ (2g)!.\]
 Let $\gamma$ be the generator of $R^{2g}\pi_*\Z_{p}(g)$ defined by the zero section of $A$; since $\operatorname{Tr}_m$ fixes $\gamma$, we have that  \[N_m \cdot  \phi_{2g}(\gamma)=N_m \cdot (2g)! (\gamma).\] 
Thus, $F_m \circ \psi(D_c)=N_m \cdot (2g)!  \phi(D_c)=0,$ because $ \phi(D_c)=0$ in $R^{2g}\pi_*\Z_{p}(g)(S)$. \newline Since $\operatorname{Tr}_m(D_c)=D_c$, $F_m$ acts on $D_c$ as multiplication by $N_m$. Hence, applying $F_m$ again, we get a canonical lift of $F_m(N_mD_c)=N_m^2D_c$ in  $H^{2g-1}_{\et}(U_c,\Z_{p}(g))$. 

\end{proof}
\begin{remark}\label{constantinvertible}
If the prime $p$ is sufficiently large, we can choose $m$ such that $N_m$ is invertible in $\Z_{p}$; it suffices to assume that $p > 2g+1$ and choose $m$ to be a generator of $(\Z/p\Z)^*$. 
\end{remark}
\begin{remark} The \'etale Gysin sequence is equivariant for the action of $\operatorname{Tr}_r$, thus we have a map \[ H^{2g-1}_{\et}(U_c,\Z_{p}(g))^{\operatorname{Tr}_r=1} \longrightarrow H^{0}_{\et}(A[c],\Z_{p})^{\operatorname{Tr}_r=1}.\] The dependence of Faltings' classes on the choice of $m$ arises from the obstruction to construct a section of this map if $g >1$.
\end{remark}
Finally, we can define classes over the base of the abelian scheme. 
\begin{definition}\label{defclasses}
Denote by $_cz_m^A$ the canonical lift of $N_m^2(D_c)$ to $H^{2g-1}_{\et}(U_c,\Z_{p}(g))$. Moreover, for any torsion section $x \in A(S)$ of order prime to $c$, we define the \'etale Eisenstein classes for $A/S$ as $$_cz_{m,x}^A := x^*{_cz_m^A}\in H^{2g-1}_{\et}(S,\Z_{p}(g)).$$ 
\end{definition}
Let $res_{\et}$ denote the (residue) map from the \'etale Gysin sequence  \[res_{\et}: H^{2g-1}_{\et}(U_c,\Z_{p}(g)) \longrightarrow H^{0}_{\et}(A[c],\Z_{p}),\] then \[res_{\et}({_cz_m^A})=N_m^2(D_c).\]
To ease the notation, we won't make explicit reference to $A$ and denote the classes by $_cz_{m,x}$ unless needed.

\subsection{Properties of the \'etale Eisenstein classes}
In the following subsection, we investigate the properties that the Eisenstein classes defined above satisfy. In particular, we show that they satisfy the two properties P1,P2 we mentioned above.
\begin{prop}\label{tracerel} Let $r$ be an integer coprime with $c$ and denote by $\operatorname{Tr}_r$ the trace map associated to the multiplication by $r:A\smallsetminus A[rc]\longrightarrow A\smallsetminus A[c],$ then $$\operatorname{Tr}_r({_cz_m})={_cz_m}.$$
\end{prop}
\begin{proof}
The trace map $\operatorname{Tr}_r$ commutes with the operator induced by multiplication by $N_m$, hence the result follows from noticing that $\operatorname{Tr}_r$ fixes $D_c$.
\end{proof}
Notice that we can immediately extend this result for trace maps associated to isogenies whose degree is coprime with $c$.
\begin{cor}
Let $h:A \to A'$ be an isogeny over $S$ of degree prime to $c$, then $$\operatorname{Tr}_h({_cz_m^A})={_cz_m^{A'}}.$$
\end{cor}
Now, we prove the compatibility property P2 we stated above.
\begin{prop} Let $c_1,c_2$ be integers, then
$$_{c_1\cdot c_2}z=[c_1]^*({_{c_2}z_m})+c_2^{2g}({_{c_1}z_m})=[c_2]^*({_{c_1}z_m})+c_1^{2g}({_{c_2}z_m})\in H^{2g-1}_{\et}(U_{c_1c_2},\Z_{p}(g)).$$
\end{prop}
\begin{proof}
We are reduced to studying what happens at the level of the characteristic classes of the $0$-cycles in $H^0_{\et}(A[c_1c_2],\Z_{p})$ we are lifting. Indeed, note that \begin{align*} [c_1]^*(c_2^{2g}e_*(1)-\pi_{|_{A[c_2]}}^*(1))+c_2^{2g}(c_1^{2g}e_*(1)-\pi_{|_{A[c_1]}}^*(1))&=D_{c_1c_2}+c_2^{2g}[c_1]^*e_*(1)-c_2^{2g}\pi_{|_{A[c_1]}}^*(1)= \\ &=D_{c_1c_2} \end{align*} hence $[c_1]^*({}_{c_2}z)+c_2^{2g}{}_{c_1}z_m$ and $_{c_1c_2}z_m$ have same residue and hence are equal.
\end{proof}
In the same fashion of the previous propositions, we can prove that the Eisenstein classes are invariant under base-change. For Siegel units, this property is well-known (e.g. \cite{kato}, Proposition 1.3(4) or \cite{scholl}, Theorem 1.2.1$(ii)$). 

\begin{prop}\label{basechrel} For any morphism $S' \to S$ and abelian scheme $A/S$, then $$g^*({_cz_m^A})={_cz_m^{A'}},$$ where $g:A'\to A$ is the base-change of $A$ to $S'$.
\end{prop}
\begin{proof}
The result follows from the fact the base change of $D_c^A$ by is $D_c^{A'}$.
\end{proof}
We can now discuss the comparison between the image under the \'etale regulator of Siegel units and the \'etale classes defined in the previous section.
\begin{prop} Let $\pi:E\to S$ be an elliptic curve and let $N_m$ be as above; fix $c$ coprime with $6p$. Then, we have that $$N_m^2({_c\theta^{\et}_E})={_cz_{m}^{E}} \in H^1_{\et}(E \smallsetminus E[c],\Q_{p}(1)),$$ where ${_c\theta^{\et}_E}$ and ${_cz_{m}^{E}}$ are the classes defined respectively in Proposition \ref{prop2.7} and Definition \ref{defclasses}.
\end{prop} 
\begin{proof}

By the compatibility of \'etale and motivic Gysin sequences with the \'etale regulator, we have that the (\'etale) characteristic class of $res({_c\theta}_E)$ is equal to $res_{\et}({_c\theta^{\et}_E})$ (see also Lemma \ref{ressiegel}). Thus, by Theorem \ref{theo1}, $N_m({_c\theta^{\et}_E})$ has residue $N_mD_c$. Apply now the operator $F_m$ to $N_m({_c\theta^{\et}_E})$; $F_m$ acts as multiplication by $N_m$ on ${_c\theta^{\et}_E}$, since $\operatorname{Tr}_m$ fixes ${_c\theta^{\et}_E}$ (by Proposition \ref{prop2.7}). Thus, $N_m^2({_c\theta^{\et}_E})$ has residue $N_m^2D_c$. It follows that the difference $N_m^2({_c\theta^{\et}_E})-{_cz_{m}^{E}}$ lies in the kernel of the residue map and thus comes from an element \[t_m \in  H^1_{\et}(E,\Z_{p}(1)),\] which is fixed by $\operatorname{Tr}_r$, for any $r$ coprime to $c$. We are left to show that $t_m$ is torsion. Indeed, after tensoring by $\Q_p$, the Leray spectral sequence gives an isomorphism \[  H^1_{\et}(E,\Q_{p}(1))= H^1_{\et}(S,R^0\pi_*\Q_{p}(1)) \oplus H^0_{\et}(S,R^1\pi_*\Q_{p}(1)), \] where $\operatorname{Tr}_r$ acts as multiplication by $r^2$ and by $r$ on each summand.  Thus $t_m=0$ in $H^1_{\et}(E,\Q_{p}(1))$. 
\end{proof}

\subsection{Compatibility in the mira-Klingen tower} \label{miraklingen}

In the following we study how the \'etale Eisenstein classes of symplectic Shimura varieties vary in the tower of levels $U_1(N)$; fix an integer $D$ coprime with $N$, and consider any sufficiently small open compact subgroup \[ K_D=\GSp_{2g}(\widehat{\Z}^{(D)}) \cdot \prod_{p \mid D}K^{p}_D \subset \GSp_{2g}(\widehat{\Z}).\]  We define the following congruence subgroup. 
\begin{definition}\label{newdeflev} Let $U_N=\prod_{p}U^{p}_N \subset \GSp_{2g}(\widehat{\Z})$ be the open compact subgroup defined by  \[U_N^{p}:=\begin{cases} U_1(p^{v_{p}(N)}) & \text{ if }p \mid N \\ K_D^{p} & \text{ if } p \mid D \\ G(\Z_{p}) & \text{ otherwise.}\end{cases}\]
\end{definition}

Since $U_N$ is sufficiently small, the Shimura variety $Y_{\GSp_{2g}}(U_N)$ is the fine moduli space of isomorphism classes of quartuples $(A,\lambda,\alpha_D,P)$, where $A$ is an abelian scheme of relative dimension $g$, $\lambda$ is a principal polarisation on $A$, $\alpha_D$ is a symplectic level $K_D$-structure and $P$ is a point of exact order $N$ of $A$. Its universal abelian scheme $\mathcal{A}_g(U_N) \to Y_{\GSp_{2g}}(U_N)$ comes equipped with the universal $N$-torsion section $e_{N}$. 
 Let now $\ell$ be a prime which does not divide $D$. We show how the \'etale Eisenstein classes of Definition \ref{defclasses} behave under the trace $\operatorname{Tr}_{\pi_\ell}$ of \[ \pi_\ell: Y_{\GSp_{2g}}(U_{N\ell})   \longrightarrow Y_{\GSp_{2g}}(U_N), \; \; (A,\lambda, \alpha_{D}, P) \mapsto (A,\lambda, \alpha_{D}, [\ell]P) \] where $[\ell]$ denotes multiplication-by-$\ell$. To do so, we prove the following.
\begin{lemma}\label{caseatp} We have the following.  \begin{enumerate}
\item Suppose that $\ell \nmid N$. If $x\in \mathcal{A}_g(U_N)(Y_{\GSp_{2g}}(U_N))$ is a section of order prime to $\ell$, we have
\begin{equation}\label{diagkl1}
\xymatrix{ Y_{\GSp_{2g}}(U_N) \sqcup Y_{\GSp_{2g}}(U_{N\ell}) \ar[rrr]^-{(x,f)}\ar[d]^{(id,\pi_\ell)}& & &\mathcal{A}_g(U_N) \ar[d]^{[\ell]} \\ 
Y_{\GSp_{2g}}(U_N) \ar[rrr]^-{\ell x} & &  & \mathcal{A}_g(U_N)
}
\end{equation}
is Cartesian, where $f:(A,\lambda, \alpha_D,P)\mapsto (A, \lambda, \alpha_D, x_A+P)$, with $x_A$ equal to the base-change of $x$ to $A$.
\item In the case when $\ell \mid N$, the diagram

\begin{equation}\label{diagkl2}
\xymatrix{Y_{\GSp_{2g}}(U_{N\ell})  \ar[rrr]^-{(id \times \pi_\ell)\circ e_{N\ell}}\ar[d]^{\pi_\ell}& & & \mathcal{A}_g(U_N) \ar[d]^{[\ell]} \\ 
Y_{\GSp_{2g}}(U_N)  \ar[rrr]^-{e_{N}} & & & \mathcal{A}_g(U_N)
}
\end{equation}
is Cartesian.
\end{enumerate}
\end{lemma}
\begin{proof}
The proof is similar to the one in the case of modular curves, as in \cite[Lemma 2.3.1]{scholl}. Recall that the universal abelian scheme $\mathcal{A}_g(U_N)/ Y_{\GSp_{2g}}(U_N)$ has a moduli parametrisation as well; indeed, for a $\Q$-scheme $S$, $\mathcal{A}_g(U_N)(S)$ is the set of pairs $(a,P)$, where $a\in Y_{\GSp_{2g}}(U_N)(S)$ and $P$ is a point in the abelian scheme $A_a/S$ defined by $a$. We check that the diagram is Cartesian by proving it for the $S$-points (i.e. in the category of Sets). This follows from the following observations.
Denote by $\bullet_a$ the base-change of $\bullet \in \mathcal{A}_g(U_N)(Y_{\GSp_{2g}}(U_N))$ to $A_a$.
\begin{enumerate}
\item Let $\ell \nmid N$. The morphism $e_N$ sends $a \in Y_{\GSp_{2g}}(U_N)(S)$ to $(a, e_{N,a})\in \mathcal{A}_g(U_N)(S).$ The pre-image of $e_{N,a}$ under multiplication by $\ell$ has either the exact order $N$ or $N\ell$. In the first case, the pre-image of $(a,e_{N,a})$ is of the form $(a, r e_{N,a})$, with $r$ the inverse of $\ell$ modulo $N$, hence it defines a point in $Y_{\GSp_{2g}}(U_N)(S)$. In the second, it will be of the form $(a, r e_{N,a}+y)$, for $y$ point of exact order $\ell$ of $A_a$. Using the universal property of $e_{N\ell }$, we get a $S$-point of $Y_{\GSp_{2g}}(U_{N\ell})$. 
 \item As above, the morphism $e_{N}$ sends $a \mapsto (a,e_{N,a})\in \mathcal{A}_g(U_N)(S)$.  Its pre-image under multiplication by $\ell$ has necessarily exact order $N\ell$, thus it defines an $S$-point of $Y_{\GSp_{2g}}(U_{N\ell})$. 
\end{enumerate} 
\end{proof}

\begin{definition}
Denote by $\Eis_{m,N}^g$ the pull-back of the class ${_cz_{m}^{\mathcal{A}_g(U_N)}}$ by the universal $N$-torsion section $e_{N}$:
\[ \Eis_{m,N}^g \in H_{\et}^{2g-1}(Y_{\GSp_{2g}}(U_N) ,\Z_{p}(g)).\]  
\end{definition}
 We now state the trace compatibility relations under the trace $\operatorname{Tr}_{\pi_\ell}$.
\begin{prop}\label{tracecompfalt} Let $c$ be an integer coprime with $6Np\ell$. We have \[\operatorname{Tr}_{\pi_\ell}(\Eis_{m,N\ell}^g) = \begin{cases} \Eis_{m,N}^g  & \text{ if }\ell \mid N; \\ 
\Eis_{m,N}^g- (\begin{smallmatrix} r I & 0 \\ 0 & r I \end{smallmatrix})^*(\Eis_{m,N}^g) & \text{ if } \ell \nmid N; \end{cases}\]
 where $r$ denotes the inverse of $N$ modulo $\ell$.
\end{prop}
\begin{proof}
The computation follows by using Lemma \ref{caseatp} and the fact that the \'etale Eisenstein classes are invariant under trace maps $\operatorname{Tr}_{\ell}$ and under base-change (since the \'etale Gysin sequence enjoys this property). For instance, if $\ell \mid N$, the Cartesianness of the diagram \eqref{diagkl2} gives \begin{align*} Tr_{\pi_\ell}(\Eis_{m,N\ell}^g) =& Tr_{\pi_\ell}\left(e_{N\ell}^*\left( {_cz_{m}^{\mathcal{A}_g(U_{N\ell})}}\right) \right) \\ =& Tr_{\pi_\ell}\left(e_{N\ell}^*\left( (id\times \pi_\ell)^*\left({_cz_{m}^{\mathcal{A}_g(U_N)}}\right)\right) \right)  \quad& \text{  Prop. \ref{basechrel}} \\  
=& e_N^*\left(\operatorname{Tr}_\ell\left({_cz_{m}^{\mathcal{A}_g(U_{N})}}\right)\right) \quad &\text{   \eqref{diagkl2}}\\ =& e_N^*\left({_cz_{m}^{\mathcal{A}_g(U_N)}}\right) \quad& \text{  Prop. \ref{tracerel}}\\ =& \Eis_{m,N}^g \end{align*}  In the case of $N$ coprime with $\ell$, we use the Cartesianness of the diagram \eqref{diagkl1} for $x=re_N$. 
\end{proof} 
This proposition is the higher dimension analogue to a result of Kato for Siegel units which appears in the proof of Propositions 2.3 and 2.4 of \cite{kato}.  

\begin{remark}
The same method gives trace-compatibility relations for the Eisenstein classes of \cite{padicinterpolation} in the cohomology group $H_{\et}^{2g-1}(Y_{\GSp_{2g}}(U_N), {\rm TSym}^k (\mathcal{H}_{\Z_{p}})(g))$, with coefficients in the $k$-th symmetric tensors of the $p$-adic Tate module of  $\mathcal{A}_g(U_{N})$. We refer to \cite[\S 3.4]{phdthesis} for further details.
\end{remark}

As a corollary of Proposition \ref{tracecompfalt}, we get a trace-compatibility relation for the push-forward of these classes to any Shimura variety $Y_G$ such that there is an embedding $\iota: \GSp_{2g} \hookrightarrow G$ which induces a morphism of Shimura varieties $Y_{\GSp_{2g}} \longrightarrow Y_G$. Let $V_{N}$ be a sufficiently small open compact subgroup of $G(\widehat{\Z})$, whose pull-back to $\GSp_{2g}$ is equal to $U_N$ and such that the map \[ \iota_{N}:Y_{\GSp_{2g}}(U_{N}) \longrightarrow Y_G(V_{N})\] is a closed immersion of codimension $d$, then we have the following.
\begin{cor}\label{rel1g} Let $\pi_\ell'$ denote the natural degeneracy morphism $Y_G(V_{N\ell})\to Y_G(V_{N})$; then, we have \[Tr_{\pi_\ell'}(\iota_{N\ell,*}(\Eis_{m,N\ell}^g)) =\begin{cases} \iota_{N,*}(\Eis_{m,N}^g)  & \text{ if }\ell \mid N; \\ 
\left( id - \iota (\begin{smallmatrix} r I & 0 \\ 0 & r I \end{smallmatrix})^* \right) \iota_{N,*}(\Eis_{m,N}^g) & \text{ if }\ell \nmid N, \end{cases}\]
 where $r$ denotes the inverse of $N$ modulo $\ell$.
\end{cor}
\begin{proof}
It follows from Proposition \ref{tracecompfalt} and the commutativity of the diagram \[ \xymatrix{ Y_{\GSp_{2g}}(U_{N\ell}) \ar[r]^-{\iota_{N\ell}} \ar[d]_{\pi_\ell} & Y_G(V_{N\ell}) \ar[d]^{\pi_\ell'} \\ 
 Y_{\GSp_{2g}}(U_{N}) \ar[r]^-{\iota_{N}}  & Y_G(V_{N}).}\]  
\end{proof}
In Section 5, we use these trace compatibility results in the case where $\iota:\GSp_4\hookrightarrow \operatorname{GU}(2,2)$, where $\operatorname{GU}(2,2)$ denotes the group scheme over $\Z$ associated to a unitary group of signature $(2,2)$.

\section{Iwasawa theory of Eisenstein classes}\label{iwtheisclass}
In the following we define classes in \'etale cohomology groups with coefficients in certain Iwasawa sheaves very close in spirit to those defined by Kings in \cite{ellipticsoule} and \cite{padicinterpolation} and \cite{KLZ} and compare the two constructions.

In the remaining part of the section, we fix an abelian scheme $A$ of relative dimension $g$ over a base $S$ of finite type over $\Z$, and a prime $p$ ($p\nmid c$). We consider $A_r:=A$ as a cover of $A$ under multiplication for $p^r$ and define the inverse system of \'etale lisse sheaves \[ \mathcal{L} := \left([p^r]_*(\Z/p^r\Z)  \right )_{r\geq 0}, \] with transition map $[p^r]_*(\Z/p^r\Z)\to [p^{r-1}]_*(\Z/p^{r-1}\Z)$ given by composition of trace map and reduction modulo $p^{r-1}$. 
Since $S$ is of finite type over $\Z$, by Remark \ref{nicepropofcohom}, we have \[H^{2g-1}_{\et}(A \smallsetminus A[c],\mathcal{L}(g)) \simeq \varprojlim_r H^{2g-1}_{\et}(A_r \smallsetminus A_r[p^rc],\Z/p^r\Z(g)). \] 
We can construct an element in this inverse limit by using the \'etale classes $_cz_m^{A_r}$. Indeed, consider the image of the element $i_r({}_cz_m^{A_r})$ in $H^{2g-1}_{\et}(A_r \smallsetminus A_r[p^rc],\Z/p^r\Z(g))$ under the natural map  \[ i_r:H^{2g-1}_{\et}(A_r \smallsetminus A_r[c],\Z/p^r\Z(g)) \longrightarrow H^{2g-1}_{\et}(A_r \smallsetminus A_r[p^rc],\Z/p^r\Z(g)). \] Since the elements ${}_cz_m^{A_r}$ are trace compatible, i.e. ${\rm Tr}_p({}_cz_m^{A_r})={}_cz_m^{A_{r-1}}$, we can define \[_c\mathcal{Z}_m :=(i_r({}_cz_m^{A_r}))_r \in  H^{2g-1}_{\et}(A \smallsetminus A[c],\mathcal{L}(g)).\]
\subsection{The integral polylogarithm of Kings}
We now recall the construction by Kings of the integral \'etale polylogarithm class \[ _c\text{pol} \in \varprojlim_r H^{2g-1}_{\et}(A_r \smallsetminus A_r[p^rc],\Z/p^r\Z(g)),\] which is characterised by having residue $D_c$. Consider a torsion section $t:S \to A$ and let $A[p^r]\langle t \rangle$ be the $A[p^r]$-torsor defined by the Cartesian diagram \[ \xymatrix{ A[p^r]\langle t \rangle \ar[r] \ar[d]_{p_{r,t}} & A_r \ar[d]^{[p^r]} \\ S \ar[r]^t & A.} \]
Define the sheaf of Iwasawa modules $\Lambda( \mathscr{H}\langle t \rangle)=\left (\Lambda_r( A[p^r]\langle t \rangle)\right)_r:= \left ( p_{r,t,*}\Z/p^r\Z \right)_r$. When $t$ is the unit section $e:S \to A$, denote $\Lambda( \mathscr{H}\langle e \rangle)$ by $\Lambda( \mathscr{H})$. The stalk at a geometric point $\bar{s}$ of $\Lambda_r( A[p^r]\langle t \rangle)$ is isomorphic to the space of $\Z / p^r\Z$-valued measures on $A[p^r]\langle t \rangle_{\bar{s}}$, hence the stalk at $\bar{s}$ of  $\Lambda( \mathscr{H}\langle t \rangle)_{\bar{s}}$ is the Iwasawa algebra $\varprojlim \Lambda_r( A[p^r]\langle t \rangle_{\bar{s}})$, which motivates the chosen notation. Moreover, $\Lambda( \mathscr{H}\langle t \rangle)$ and $ \mathcal{L}$ are simply related by the following.
\begin{lemma}[\cite{padicinterpolation}, Lemma 6.1.2] There is a canonical isomorphism \[ t^* \mathcal{L} \simeq \Lambda( \mathscr{H}\langle t \rangle)\]
\end{lemma}
\begin{proof}
It is sufficient to work at finite level. Thus, the result follows from the Cartesian diagram above, since we deduce that \[t^*[p^r]_* \Z/p^r\Z \simeq p_{r,t,*}\Z/p^r\Z. \] \end{proof}
We are now ready to state a result of Kings, which allows to define the integral polylogarithm class in terms of its residue; this is of fundamental importance in the study of integrality of Eisenstein classes and makes explicit the relation between polylogarithm classes and Siegel units in the case of elliptic curves (as in \cite{ellipticsoule}, Theorem 12.4.21).
\begin{proposition}[\cite{padicinterpolation}, Proposition 6.3.1]
We have a short exact sequence \[ \xymatrix{ 0 \ar[r] &  H^{2g-1}_{\et}(A \smallsetminus A[c], \mathcal{L}(g)) \ar[r]^-{res} & H^0_{\et}(S,\pi_{|_{A[c]},*}\pi_{|_{A[c]}}^*\Lambda( \mathscr{H})) \ar[r] & H^0(S,\Z_p) \ar[r] &0} \]
\end{proposition}
\begin{proof}
The statement differs from the one of \cite[Proposition 6.3.1]{padicinterpolation}, since we have used that \[ H^{2g-1}_{\et}(A \smallsetminus A[c], \mathcal{L}(g)) \simeq H^{2g-1}_{\et}(S, R\pi_{*}Rj_{c*}j_{c}^{*}\mathcal{L}(g)),\] as explained in \cite[Section 6.5]{padicinterpolation}.\end{proof}
Let $D_c$ be the \'etale characteristic class of $c^{2g}(0)-A[c]$, as defined in Definition \ref{dc}. Since \[D_c \in H^0_{\et}(S,\pi_{|_{A[c]},*}\pi_{|_{A[c]}}^*\Lambda( \mathscr{H})),\] we can define the integral \'etale polylogarithm class as follows.
\begin{definition} The integral \'etale polylogarithm class associated to $D_c$ is \[ \operatorname{_{c}pol}\in H^{2g-1}_{\et}(A \smallsetminus A[c], \mathcal{L}(g)) \] such that $res( \operatorname{_{c}pol})=D_c$. \end{definition}
\begin{remark}
There is a clear discrepancy between the construction of Kings and the one of Faltings. The \'etale polylogarithm class does not depend on the choice of an auxiliary integer $m$, because of the vanishing of the direct image sheaves of $\mathcal{L}$ in degree less than $2g$ (\cite[Theorem 6.2.3]{padicinterpolation}). 
\end{remark}

\subsection{The comparison of the two constructions}

The comparison between the integral polylogarithm and the class of Theorem \ref{thetheorem} relies on the very characterisation of the two elements, as the next proposition shows.  
\begin{proposition}\label{comparisonfk}
Let $c,m$ be as before. \[N_m^2(\operatorname{_{c}pol})={}_c\mathcal{Z}_m.\]
\end{proposition}
\begin{proof}
This is a straight-forward consequence of the comparison of the residues at finite levels. Indeed, both $N_m^2(\operatorname{_{c}pol})$ and ${}_c\mathcal{Z}_m$ are inverse system of classes having residues $N_m^2D_c$ in $H^0(A[cp^r],\Z/p^r\Z)$.
\end{proof}
\begin{remark}  
This is the \emph{integral counterpart} of a comparison between the underlying rational motivic classes of Faltings and Kings, which follows from the description of the degree 0 part of the polylogarithm on abelian schemes of \cite{kingsrossler} (e.g. \cite[\S 3.2.4]{phdthesis}). 
\end{remark}

\begin{definition}\label{definitionlambdaadicclass}
 Let $t:S \to A$ be a $q$-torsion section, for $q$ coprime to $c$; define the $\Lambda$-adic Eisenstein class $_c\mathcal{EI}_m(t)$ to be  \[ t^*({_c\mathcal{Z}_m}) \in H^{2g-1}_{\et}(S, t^*\mathcal{L}(g)) \simeq \varprojlim_r H^{2g-1}_{\et}(A[p^r]\langle t \rangle,\Z/p^r\Z(g)).  \] 
\end{definition} 
 \begin{remark} By Proposition \ref{comparisonfk}, $_c\mathcal{EI}_m(t)$ is just a multiple of Kings' Eisenstein-Iwasawa class \[_c\mathcal{EI}(t):=t^*(\operatorname{_{c}pol}).\]
 \end{remark}
 
  \begin{remark}
Kings defines moment maps (e.g. \cite[Definition 5.5.2]{padicinterpolation}) \[mom_k:H^{2g-1}_{\et}(S,\Lambda( \mathscr{H})(g)) \longrightarrow H^{2g-1}_{\et}(S, \Gamma_k( \mathscr{H})(g)),\] for any integer $k \geq 0$, where $\Gamma_k( \mathscr{H})$ is the sheaf of $k^{th}$ divided powers of the $p$-adic Tate module of $A$. This allows us to define classes \[ _cz_{m,q}^k:=mom_k(\operatorname{Tr}_q({}_c\mathcal{EI}_m(t))),\] which, after tensoring with $\Q_p$, are in $H^{2g-1}_{\et}(S, Sym^k( \mathscr{H}_{\Q_p})(g)).$ However, this element is just a multiple of the Eisenstein class constructed in \cite{padicinterpolation}.
 \end{remark}
\subsection{$\Lambda$-adic Eisenstein classes for $\GSp_{2g}$}

In the following, we restrict our attention to the case where $S$ is an integral model of the Shimura variety for $\GSp_{2g}$.

 Let  $U_{p^r}$ be the sufficiently small open compact subgroup of  $\GSp_{2g}(\widehat{\Z})$ introduced in  Definition \ref{newdeflev}, and let $\mathcal{Y}_{\GSp_{2g}}(U_{p^r})$ be the model of ${Y}_{\GSp_{2g}}(U_{p^r})$ over $\Z[\frac{1}{Dp}]$. Denote by $\mathcal{A}(U_{p^r})/\mathcal{Y}_{\GSp_{2g}}(U_{p^r})$ its universal abelian scheme, which is equipped with the universal $p^r$-torsion section $e_{p^r}$. Associated to $\mathcal{A}(U_{p^r})$ and $e_{p^r}$, we have the \'etale Eisenstein class  \[ \Eis_{m,p^r}^g:= e_{p^r}^*\left({_cz_{m,e_{p^r}}^{\mathcal{A}(U_{p^r})}}\right)\in  H_{\et}^{2g-1}(\mathcal{Y}_{\GSp_{2g}}(U_{p^r}) ,\Z_{p}(g)).\]  

\begin{lemma}\label{integralversionofA} Let $\operatorname{Tr}_{\pi_p}$ be the trace map of the natural degeneracy map $\pi_p:\mathcal{Y}_{\GSp_{2g}}(U_{p^{r+1}}) \longrightarrow \mathcal{Y}_{\GSp_{2g}}(U_{p^{r}}). $ If $r \geq 1$, then \[ \operatorname{Tr}_{\pi_p}(\Eis_{m,p^{r+1}}^g)=\Eis_{m,p^r}^g.\]  
\end{lemma}
\begin{proof}
The proof is identical to the one of Proposition \ref{tracecompfalt} for $\ell=p$. Indeed, it follows from Lemma \ref{integralmodelversion} and the Cartesianness of 
\[ \xymatrix{\mathcal{Y}_{\GSp_{2g}}(U_{p^{r+1}})  \ar[rrr]^-{(id \times \pi_p)\circ e_{p^{r+1}}}\ar[d]_{\pi_p}& & & \mathcal{A}(U_{p^r}) \ar[d]^{[p]} \\ 
\mathcal{Y}_{\GSp_{2g}}(U_{p^{r}})  \ar[rrr]^-{e_{p^r}} & & & \mathcal{A}(U_{p^r}). } \] 
\end{proof}
\begin{remark} By the previous lemma, we have \[ \Eis_{m,p^{\infty}}^g:=  (\Eis_{m,p^r}^g )_{r\geq 1} \in \varprojlim_{r\geq 1} H^{2g-1}_{\et}(\mathcal{Y}_{\GSp_{2g}}(U_{p^{r}}),\Z_p(g)),\] where the inverse limit is taken with respect to the trace maps $\operatorname{Tr}_{\pi_p}$.
\end{remark}
The class $\Eis_{m,p^{\infty}}^g$ can be directly 
related to Kings' Eisenstein-Iwasawa class $_c\mathcal{EI}(e_p):=e_p^*(\operatorname{_{c}pol})$ associated to $\mathcal{A}(U_{p})$. Indeed, we get the following generalisation of Theorem 4.5.1(1) of \cite{KLZ}.
\begin{theorem}\label{secondcomparisonfk}
There is an isomorphism \[H^{2g-1}_{\et}(\mathcal{Y}_{\GSp_{2g}}(U_p),\Lambda( \mathscr{H}\langle e_p \rangle)(g)) \simeq \varprojlim_{r\geq 1} H^{2g-1}_{\et}(\mathcal{Y}_{\GSp_{2g}}(U_{p^{r}}),\Z_p(g)), \] where the inverse limit is with respect to the trace map of the natural degeneracy map \[ \pi_p:\mathcal{Y}_{\GSp_{2g}}(U_{p^{r+1}}) \longrightarrow \mathcal{Y}_{\GSp_{2g}}(U_{p^{r}}). \] Moreover, under this isomorphism the $\Lambda$-adic class $_c\mathcal{EI}_m(e_p)$ of Definition \ref{definitionlambdaadicclass} is mapped to $\Eis_{m,p^{\infty}}^g$. In particular, we have \[ \Eis_{m,p^{\infty}}^g=N_m^2({_c\mathcal{EI}(e_p)}).\]
\end{theorem}
\begin{proof}
The proof is very similar to the one of Theorem 4.5.1(1) of \cite{KLZ}.
Let $\mathcal{A}:=\mathcal{A}(U_{p})$. Note that there is an isomorphism of schemes $f_r:\mathcal{Y}_{\GSp_{2g}}(U_{p^{r+1}}) \simeq \mathcal{A}[p^r]\langle e_p \rangle$ such that 
\[ \xymatrix{  \mathcal{Y}_{\GSp_{2g}}(U_{p^{r+1}}) \ar[r]^-{f_m} \ar[d]^{\pi_p^r} & \mathcal{A}[p^r]\langle e_p \rangle \ar[d]_{pr_{r,e_p}} \\ \mathcal{Y}_{\GSp_{2g}}(U_p) \ar@{=}[r] & \mathcal{Y}_{\GSp_{2g}}(U_p). } \]
This follows from the very definition of $\mathcal{A}[p^r]\langle e_p \rangle$. Indeed, a point of $\mathcal{A}[p^r]\langle e_p \rangle$ over $(A,\lambda,\alpha_D,P)\in \mathcal{Y}_{\GSp_{2g}}(U_p)$ is given by a point $Q$ of order $p^{r+1}$ of $A$ such that $[p^r]Q=P$. Hence, $f_r$ is the isomorphism defined by sending \[ (A,\lambda,\alpha_D,Q)\mapsto \left( (A,\lambda,\alpha_D,[p^r]Q),Q\right). \]  The isomorphism $f_r$ induces \[ H^{2g-1}_{\et}(\mathcal{A}[p^r]\langle e_p \rangle,\Z_p(g))\simeq H^{2g-1}_{\et}(\mathcal{Y}_{\GSp_{2g}}(U_{p^{r+1}}),\Z_p(g)), \] for all $r \geq 0$, and, passing to the limit, we get the desired isomorphism. \newline Moreover, under $f_r$, the morphism $\mathcal{A}[p^r]\langle e_p \rangle \longrightarrow \mathcal{A}$ corresponds to the universal $p^{r+1}$ torsion section $e_{p^{r+1}}$ of $\mathcal{Y}_{\GSp_{2g}}(U_{p^{r+1}})$. This shows that $\Eis_{m,p^{r+1}}^g$ corresponds to the restriction of ${_cz_m^{\mathcal{A}_r}}$ to $\mathcal{A}[p^r]\langle e_p \rangle$. Thus, \[{}_c\mathcal{EI}_m(e_p)= \Eis_{m,p^{\infty}}^g.\]
Finally, the last equality follows from Proposition \ref{comparisonfk}.
\end{proof}

\section{The $\operatorname{GU}(2,2)$-Shimura variety}\label{genongu22shimvar}

\subsection{The groups}
Let $H:=\GSp_4$ be the group scheme over $\Z$, which was previously defined as
\[ H(R):=\GSp_4(R) = \{ (g,m_g) \in (\GL_4 \times \GL_1)(R) : g^t J g=m_g J \}. \]  After an auxiliary choice of imaginary quadratic field $F$ with ring of integers $\O_F$, define $G := \operatorname{GU}(2,2)$ the group scheme over $\Z$ given by  
\[ G(R) := \operatorname{GU}(2,2)(R) = \{ (g,m'_g) \in \GL_4(R \otimes_{\Z} \O_F ) \times \GL_1(R) :\bar{g}^t J g = m'_g J\}, \]
where $\bar{\bullet}$ denotes the non-trivial automorphism of order 2 of $F/\Q$.
We denote by $\nu: G \longrightarrow \GL_1$, the unitary multiplier map given by $g \mapsto m'_g$. 
Of fundamental importance for these notes will be the embedding $\varphi:H \hookrightarrow G$, given by the natural embedding of $\GL_4$ inside $Res_{\O_F/\Z}\GL_4$.

\subsection{The Shimura variety}
We recall the definition of the Shimura variety attached to $G$ and its moduli interpretation. Let $F$ be the imaginary quadratic field used to define $G$ and fix an integral basis $\{ 1, y_f \}$ of $F$. The Shimura datum for $H$ defines a Shimura datum $(G,X_G)$ via the embedding $H \hookrightarrow G$. Thus, define the following. 
\begin{definition}  For an open compact subgroup  $U$ of $G(\Af)$, consider the double quotient space
\[Y_G(U)(\C) =G(\Q) \backslash X_G \times G(\Af)/ U.\] 
\end{definition}
Multiplication by $g \in G(\Af)$ gives a morphism \[g:Y_G(U)(\C) \longrightarrow Y_G(g^{-1}Ug)(\C), \quad [\tau,h]\mapsto [\tau,hg].\] This defines a right action of $G(\Af)$ on $Y_G(U)(\C)$. We note that if $U$ is neat (in the sense of \cite{pink}, p. 13), the quotient $Y_G(U)(\C)$ is the set of complex points of a smooth (quasi-projective) variety $Y_G(U)$ over $\C$. Since the reflex field of $(G,X_G)$  is $\Q$, $Y_G(U)$ admits a canonical model (smooth, of finite type) over $\Q$ (for instance, see \cite{moonen} Theorem 2.18). More precisely, one shows that the variety is canonically defined over $\bar{\Q}$ and then by using the (continuous) action of $Gal(\bar{\Q}/\Q)$ on this model over $\bar{\Q}$, it is possible to descend to $\Q$ (see \cite{moonen}, chapter 2). Moreover, we obtain $\Q$-morphisms $g:Y_G(U)\longrightarrow Y_G(g^{-1}Ug)$, for any $g \in G(\Af)$. We have a description of $Y_G(U)$ as a moduli space of abelian schemes of relative dimension 4 with principal polarisation, level $U$ structure with compatible action of $\O_F$. More precisely, consider the functor $\mathcal{L}_{U}$ from the category of locally Noetherian $F$-schemes to $Sets$, which to $S$ assigns isomorphism classes of $(A,\lambda, \iota , \alpha_{U})$, where \begin{itemize}
\item $A/S$ is an abelian scheme of relative dimension 4;
\item $\lambda$ is a principal polarisation of $A$;
\item $\iota: \O_F \to End_S(A)$ is a homomorphism which is compatible with $\lambda$, i.e. $$\iota(a)^{\vee} \circ \lambda = \lambda \circ \iota( \bar{a}),$$ for all  $a \in \O_F$, such that there is a splitting $$Lie_S(A)=Lie_S(A)^+\oplus Lie_S(A)^-,$$
where the direct summands are locally free $\O_S$-sheaves of ranks 2 and $z \in \O_F$ acts on $Lie_S(A)^+$ by $f(z)$ and on $Lie_S(A)^-$ by $f(\bar{z})$, where $f: F \to \O_S$ is the structure homomorphism;
\item $\alpha_U$ is a unitary $U$-level structure. 
\end{itemize}
\begin{remark} \leavevmode \begin{enumerate}
\item We call the level structure unitary to emphasise the difference with the $\GSp_{2g}$-case and we recall its definition below;
\item The splitting condition on $Lie_S(A)$ is equivalent to requiring that  \[det(T-\iota(z)|Lie_S(A))=(T-f(z))^2(T-f(\bar{z}))^2\in \O_S[T],\] where $f: F \to \O_S$ is the structure homomorphism coming from the $F$-structure of $S$;
\item In the case of neat open compact subgroup $U$ of $G(\widehat{\Z})$, the functor $\mathcal{L}_{U}$ is representable by a quasi-projective $F$-scheme (\cite{cptpel} Theorem 1.4.1.11) which is identified with ${Y_G(U)}_{/F}$, which is the base-change over $F$ of $Y_G(U)/\Q$.
\item Following $(3)$, the reader might wonder why we do not define an analogous functor, say $\mathcal{L}^{\Q}_{U}$, over the category of $\Q$-schemes. Note that the condition on $Lie_S(A)$ does not make sense anymore. Namely, if we have a $\Q$-vector space $V$ which does not have an $F$-structure but an action of $\O_F$ by $\Q$-linear maps, then the endomorphism corresponding to $y_F$ (where $\{1,y_F\}$ is an integral basis of $F$) is only diagonalisable after tensoring by $\O_F$ and the required decomposition would work for the $F$-vector space $V \otimes_{\Z} \O_F$. Hence the $\Q$-points of $\mathcal{L}^{\Q}_{U}$  are the empty set. Note that the same idea applies to any $\Q$-scheme $S$ and locally free $\O_S$-sheaf without $F$-structure. Since $\mathcal{L}_{U}$ is representable by ${Y_G(U)}_{/F}$, then $\mathcal{L}^{\Q}_U$ is representable by the scheme ${Y_G(U)}_{/F}\to Spec(F) \to Spec(\Q)$, obtained by composing with the inclusion $\Q \hookrightarrow F$, which is clearly not isomorphic to the canonical model $Y_G(U)/\Q$.
\end{enumerate}
\end{remark}

We now describe what a unitary level structure for an abelian scheme $A/S$ with $\O_F$-action is. Denote by $K(p)$ the kernel of the reduction modulo $p$ \[G(\widehat{\Z}) \longrightarrow G(\Z/p\Z).\] Consider the sesquilinear pairing $H$ on $\O_F^4$ given by $y_F\cdot J$ and let  $\langle \bullet, \bullet \rangle:\O_F^4 \times \O_F^4 \to \Z$ be the skew-symmetric pairing defined by \[ H(x_1,x_2)=\langle x_1, y_F\cdot  x_2 \rangle+y_F \langle x_1, x_2 \rangle .\]
Note that  $\langle \bullet , \bullet \rangle$ restricted to $\Z^4$ is the pairing defined by the matrix $J$. \newline 
As we previously did for symplectic structures, we can define level $U$-structures for general open compact subgroups of $U \subset G(\widehat{\Z})$, as follows.

\begin{definition}\label{unitarylevel}
Let $U$ be an open compact subgroup of $G(\widehat{\Z})$ and for any integer $M$ such that $K(M) \subset U$ denote by $U_M$ the quotient $U/K(M)$. Then, a unitary level $U$ structure of $(A,\lambda, \iota)_{/S}$ is a collection $\{ \alpha_{U_M} \}_{M}$, where $M$ varies among the integers such that $K(M) \subset U$, of elements $\alpha_{U_M}$ such that \begin{enumerate}
\item $\alpha_{U_M}$ is a locally \'etale defined $U_M$-orbit of an $\O_F$-equivariant isomorphism \[ \alpha_M: (\O_F/M\O_F)^4\longrightarrow A[M],\] with the property that there is an isomorphism $\beta_M:((\Z/M\Z)(1))_{/_S} \longrightarrow \mu_{M/_S}$ which makes the diagram 
  \[ \xymatrix{(\O_F/ M \O_F)^4_{/_S}{\times}_S (\O_F/M\O_F)^4_{/_S}\ar[rr]^-{\langle \bullet , \bullet \rangle} \ar[d]_{\alpha_M \times \alpha_M} & & ((\Z/M\Z)(1))_{/_S} \ar[d]^{\beta_M} \\ 
  A[M] {\times}_S A[M] \ar[rr]^{e_{\lambda}} & & \mu_{M/_S}
    }\]
    commutative.
\item If $L|M$, $\alpha_{U_L}$ correspond to the reduction modulo $L$ of $\alpha_{U_M}$.
\end{enumerate} 
\end{definition}

\begin{remark} One could define unitary level structure on the Tate module of the abelian scheme. For instance, a unitary full level $p$ structure on  $(A,\lambda,\iota)_{/_S}$ corresponds to a collection $\{\alpha_{\bar{s}} \}_{\bar{s}}$ of $\pi_1(S,\bar{s})$-invariant $K(p)$-orbit of an $\O_F$-equivariant isomorphism $$\alpha_{\bar{s}}: \hat{\O}_F^4 \longrightarrow T_{\bar{s}}(A)$$ which respects the two forms $e_{\lambda}$ and $\langle \bullet , \bullet \rangle$  as in Definition \ref{unitarylevel}, and such that $\alpha_{\bar{s}}$ and  $\alpha_{\bar{s}'}$ are canonically identified for any two geometric points $\bar{s} ,\bar{s}'$ in the same connected component.
As in the symplectic case, $g \in G(\widehat{\Z})$ acts on the isomorphism $\alpha_{\bar{s}}$ by $\alpha_{\bar{s}} \circ g$.
\end{remark}
 
\subsection{A remark on the embedding at the level of moduli}

We note that $\varphi: H \hookrightarrow G$ induces closed immersions $\phi_{U}: Y_H(U \cap H) \longrightarrow Y_G(U),$ 
for certain open compact subgroups $U$ of $G(\Af)$.
 Due to the moduli space description of these spaces, it is reasonable to ask whether we have a nice explicit description of the pull-back of the universal element of ${Y_G(U)}_{/F}$ in terms of the (base-change to $F$) of the universal element of ${Y_H(U \cap H)}$.

First, recall the main properties of Serre's tensor construction for abelian schemes. 

\begin{lemma}[Serre's Tensor Construction] Let $R$ be a ring and $M$ be a finite projective $R$-module; for any group scheme $A$ with $R$-module structure over $S$, the functor from $S$-schemes to Sets $$\mathcal{F}:M \to A(T) \otimes_R M$$ is representable by a group scheme, which is denoted by $$A\otimes_R M.$$
Moreover, in the case where $A/S$ is an abelian scheme, then we have: \begin{enumerate}
\item $A\otimes_R M$ is an abelian scheme over $S$;
\item There is a canonical isomorphism $$T_{p}(A\otimes_R M) \simeq T_{p}(A)\otimes_R M;$$
\item There is a canonical isomorphism of $\O_S$-modules $$Lie_S(A\otimes_R M)\simeq Lie_S(A)\otimes_R M.$$
\item  We have $$(A \otimes_R M)^{\vee} \simeq A^{\vee} \otimes_R M^{\vee},$$ where $A^{\vee}$ denotes the dual abelian scheme of $A$ and $M^{\vee}=Hom_R(M,R)$. 
\end{enumerate}
\end{lemma}
\begin{proof}
We only give a sketch of the proof. Note that if $M$ is a free $R$-module of rank $n$, $\mathcal{F}$ is representable by $A^n$; more generally,  take a presentation \[\xymatrix{ R^m \ar[r]^{g} &R^n\ar[r] & M^{\vee} \ar[r] & 0,}\] and apply $Hom_R(-,A(T))$ to get \[\xymatrix{ 0 \ar[r] &Hom_R(M^{\vee},A(T))\ar[r] & A(T)^n \ar[r]^-{g_T} &  A(T)^m.}\] Since $A(T) \otimes_R M \simeq Hom_R(M^{\vee},A(T))$, we conclude that $\mathcal{F}$ is representable by the kernel of $(g_T)_T$.
 By \cite{conradgrosszagier} Theorem 7.2 and Theorem 7.5, the tensor construction preserves smoothness, properness and geometric connectedness of fibres, hence if $A$ is an abelian scheme so is $A\otimes_R M$. Property $(2)$ is a direct consequence of the fact that tensoring by $M$ is left exact (since $M$ is projective), hence $A[p^n]\otimes_R M \simeq  (A\otimes_R M)[p^n]$. $(3)$ is proved similarly (see \cite{serretensor}, Lemma 3); for a proof of $(4)$, we refer to Proposition 5 of \cite{serretensor}. 
  \end{proof}
We can now prove the main result of this section. 
Let $(\Am_{H,F},\lambda_{H,F},\alpha_{H, U \cap H,F})$ be the base-change to $F$ of the universal object  of $Y_H(U\cap H)$.
\begin{proposition}\label{onmoduli}
The abelian scheme $\Am_{H,F} \otimes_{\Z} \O_F/Y_H(U\cap H)_{/F}$ (and the extra structure) is identified with the pull-back by $\phi_{U}$ of the universal object $(\Am_{G,F},\lambda_{G,F},\iota_{F},\alpha_{G,U,F})/Y_G(U)_{/F}.$
\end{proposition}

 \begin{proof}
 Consider the pull-back by  $\phi_{U}$ of $\Am_{G,F}$:
\[ \xymatrix{ ?  \ar[rrr] \ar[d] & & &\Am_{G,F} \ar[d]  \\ 
Y_H(U\cap H)_{/F} \ar[rrr]^{\phi_{U}} & & & Y_G(U)_{/F}. }\]
We would like to prove it is isomorphic to $\Am_{H,F} \otimes_{\Z} \O_F$. By the Lemma above, we have that $\Am_{H,F} \otimes_{\Z} \O_F$ is an abelian scheme over $Y_H(U \cap H)_{/F}$ of dimension 4; consider the action
$$\iota: \O_F \longrightarrow End_{Y_H(U \cap H)_{/F}}(\Am_{H,F} \otimes_{\Z} \O_F),$$ where $\iota(z)$ is given by multiplication by $z$ in the second factor (which is a well-defined endomorphism by \cite{serretensor} Proposition 2 (c)). Moreover, $\iota$ induces the required decomposition of $$\mathcal{V}:= Lie_{ Y_H(U \cap H)_{/F}}(\Am_{H,F} \otimes_{\Z} \O_F).$$ Indeed, $\mathcal{V}$ is a $\O_S \otimes_{\Z} \O_F$-module and one can consider the elements 

\begin{align*} x_1 &= f(y_F)\otimes 1 - 1 \otimes \iota(y_F), \\
x_2 &= f(\bar{y}_F)\otimes 1 - 1 \otimes \iota(y_F), \end{align*} 
where $f:F \to \O_S$ is the structure homomorphism and $\{1,y_F \}$ is an integral basis of $F$. The quotients $\mathcal{V}/(x_1)$ and $\mathcal{V}/(x_2)$ are the maximal quotients on which $\O_F$ acts respectively by the structure monomorphism and by its conjugate and since the discriminant $D_F$ is invertible in $\O_S$ we have \[\mathcal{V}=\mathcal{V}/(x_1) \oplus \mathcal{V}/(x_2).\]
We would like to show that $\lambda_{H,F}$ and the choice of $\iota$ give a $\O_F$-linear polarisation on $\Am_{H,F} \otimes_{\Z} \O_F$. This amounts to choose a $\Z$-linear isomorphism $g:\O_F \longrightarrow \O_F^{\vee}$, such that  \[ \lambda_{H,F} \otimes g:\Am_{H,F} \otimes_{\Z} \O_F \longrightarrow \Am_{H,F}^{\vee} \otimes_{\Z} \O_F^{\vee}  \simeq (\Am_{H,F} \otimes_{\Z} \O_F)^{\vee}\]
is compatible with $\iota$, in the sense that \[ (\lambda_{H,F}\otimes g) \circ \iota(\bar{z}) = \iota(z)^{\vee} \circ \lambda_{H,F} \otimes g.\]

The isomorphism $g$ is defined to be the composition on the right of $\bar{\bullet}:\O_F\to \O_F$ with the isomorphism $\O_F \to \O_F^{\vee}$ determined by the choice of an integral basis of $F$. Hence, $ \lambda_{H,F} \otimes g$ defines a polarisation (\cite{serretensor}, Theorem 17) compatible with the $\O_F$-action. Moreover, $\lambda_{H,F}\otimes g$ is an isomorphism.
We are left to show that there is a unitary level structure induced from the symplectic level structure $\alpha_{H,U \cap H, F}$. By part $(2)$ of Lemma above, there is an isomorphism \[ (\Am_{H,F} \otimes_{\Z} \O_F)[p^m] \simeq \Am_{H,F}[p^m] \otimes_{\Z} \O_F,\] hence a symplectic full $p^m$-level structure on $\Am_{H,F}$ induces a unitary full $p^m$-level structure on $$\Am_{H,F} \otimes_{\Z} \O_F.$$ Indeed, take a geometric point $s$ of $Y_H(U \cap H)_{/F}$; we consider the base-change of the symplectic level structure to $s$, so that we get a (symplectic) isomorphism \[ \alpha_s:(\Z/p^m\Z)^4 \longrightarrow \Am_{H,F,s}[p^m].\] 
Since the tensor product construction commutes with base-change, we have an isomorphism $$\alpha_s \otimes \O_F:(\O_F/p^m\O_F)^4 \longrightarrow (\Am_{H,F} \otimes \O_F)_s[p^m],$$ which is $\O_F$-equivariant and respects the form given by $ \langle \bullet , \bullet \rangle$ and the form given by composition of $\lambda_{H,F}\otimes g$ and Weil pairing, since $\iota$ coincide with the $\O_F$-module structure of $(\O_F/p^m\O_F)^4$. Thus, it follows that a $U \cap H$-level symplectic structure induces a $U$-level unitary structure.
Summing all, we constructed a point $\psi \in Y_G(U)_{/F}(Y_H(U\cap H)_{/F})$. In particular, for any locally Noetherian $F$-scheme $S$, $\psi(S)$ is described
by sending the isomorphism class of $(A,\lambda,\eta_{U\cap H})/S$ to the isomorphism class of $(A\otimes_{\Z} \O_F, \lambda \otimes g, \iota, \eta_{U})/S$. 
This morphism corresponds uniquely to a $Gal(\bar{\Q}/F)$-equivariant morphism $\psi_{\bar{\Q}}$ of $Y_H(U \cap H)_{/\bar{\Q}} \to Y_G(U)_{/\bar{\Q}}$  (\cite{moonen} 2.15.1), which is equal to the descent to $\bar{\Q}$ of $\phi_{U}$. This can be checked by evaluating the two morphisms on the set of special points, which forms a Zariski dense set of $Y_H(U \cap H)(\C)$. Recall that each special point of $(h,g) \in Y_H(U \cap H)(\C)$ is associated to a $\Q$-torus $T$ in $H$ with the property that $h\in X_H$ factors through $T_{/\R}$ (here $(H,X_H)$ is the Shimura datum of $Y_H$). We denote it by $s_T$. Then, $\phi(s_T)$ is the special point of $Y_G(U)(\C)$, associated to $T$, where $T$ is seen inside $G$ via $\varphi:H \to G$.  On the other hand, $s_T$ corresponds to (an isomorphism class of) a CM abelian variety $A=A_T$ with polarisation $\lambda$ and symplectic level $U \cap H$ structure $\eta_{U \cap H}$. Thus, \[ \psi(s_T)=(A \otimes_{\Z} \O_F,\iota, \lambda \otimes g, \eta_U). \]
This coincides with $\phi(s_T)$, which corresponds to the $\O_F$-ification of the polarised Hodge structure corresponding to $A$. This completes the proof, since the equality $\psi_{\bar{\Q}}=\phi_{U,\bar{\Q}}$ implies that $\phi_U$  and $\psi$ coincide as morphisms of the models of the Shimura varieties over $F$.
\end{proof}

\begin{corollary} The morphism $\phi_{U}:Y_H(U \cap H)_{/F} \longrightarrow Y_G(U)_{/F}$ is given by sending the $S$-point $(A,\lambda,\eta)\in Y_H(U \cap H)_{/F}(S)$ to the $S$-point $(A\otimes_{\Z} \O_F, \lambda \otimes g, \iota, \eta') \in Y_G(U)_{/F}(S)$, where $\iota:\O_F \longrightarrow End_S(A\otimes_{\Z} \O_F)$, $\lambda \otimes g$ and $\eta'$ are defined as in the proof of Proposition \ref{onmoduli}.
\end{corollary}
\begin{remark} Proposition \ref{onmoduli} and its proof also work in the case where $\phi$ is the morphism of Shimura varieties associated to the natural embedding $\GSp_{2n} \hookrightarrow \operatorname{GU}(n,n)$.
\end{remark}
\section{Trace compatible classes for $\operatorname{GU}(2,2)$}\label{twovarfamily}
In the following we explain how to construct \'etale cohomology classes for the Shimura variety $Y_G$ starting from Eisenstein classes for $H$. We prove that the resulting classes are trace compatible with respect to a two variable family of level subgroups of $G(\widehat{\Z})$.

\subsection{Compatibility in the mira-Klingen tower}
Associated to the closed immersion $\phi_{V}:Y_H(V \cap H) \longrightarrow Y_G(V),$ (for good $V$) there is the \'etale Gysin map  \[ \phi_{V,*}:H^3_{\et}(Y_H(V \cap H), \Z_{p}(2)) \longrightarrow  H^{5}_{\et}(Y_G(V), \Z_p(3)). \]  
We show that the push-forward under these maps of the Eisenstein classes for $H$ is compatible in the mira-Klingen tower at $p$ for $G$. Recall the following.
\begin{lemma} \label{kisin}
Let $V_{q} \subset G(\Q_{q})$ be a sufficiently small open compact subgroup and let $U=U_{q}U^{({q})}$ be an open compact subgroup of $H(\Af)$ such that $U_{q}= V_{q} \cap H(\Q_{q})$; then \begin{enumerate} 
\item there exists a compact open subgroup $V=V_{q}V^{({q})}$ of $G(\Af)$ with $U \subset V$ and such that $\varphi$ induces a closed immersion \[Y_H(U) \hookrightarrow Y_G(V);\] 
\item let $p \ne q$ be any prime such that $V^{(q)}$ has trivial component at $p$; for any open compact $V_{p}\subset G(\Q_{p})$ with $U_p= V_p \cap H(\Q_p)$, the morphism \[Y_H(U_{q}U_{p}U^{(q)}) \longrightarrow Y_G(V_{q}V_{p}V^{(q)}) \] is still a closed immersion.
\end{enumerate}
\end{lemma}
\begin{proof}
$(1)$ is a particular case of \cite{kisin}, Lemma 2.1.2. Now, suppose that $z,z' \in Y_H$ have same image in $Y_G(V_{q}V_{p}V^{(q)})$, i.e. there is $v \in V_{q}V_{p}V^{(q)}$ such that $z=z' \cdot v$. We claim that $v$ lies in $U_{q}U_{p}U^{(q)}$. \newline Indeed, since $z,z'$ have same image in $Y_G(V_{q}V_{p}V^{(q)})$, they map to the same element in $Y_G(V)$. Thus, by $(1)$, there exists $u \in U$ such that $z=z' \cdot u$. Hence, \[z=z' \cdot u=z' \cdot v\] implies that $vu^{-1}$ fixes $z'$. Since $V$ acts without fixed points, we have that $vu^{-1}=1$. In particular, we conclude that $v \in H$ and, consequently,  $v \in U_{q}U_{p}U^{(q)}$, which completes the proof of $(2)$.  \end{proof}
 Consider the following subgroups of $G(\Z_p)$.
\begin{definition} For any integer $r \geq 1$, define the subgroup $V_1(p^r) \subset G(\Z_p)$ as follows: \begin{eqnarray}
V_1(p^r) &:=  &\{ M \in  G(\Z_p) | R_{4}(M) \equiv (0, \cdots, 0,1) \text{ mod }p^r\O_F \} 
\end{eqnarray} 
where $R_i(M)$ denotes the $i$-th row of $M$.
If $N= \prod p_i^{e_i}$, then $V_1(N) \subset G(\widehat{\Z})$ is defined to be the subgroup of elements $(g_p)_p$ such that $g_{p_i}\in V_1(p_i^{e_i})$.
\end{definition}

Note that $V_1(p^r) \cap H(\Q_p)=U_1(p^r)$, which was defined in Definition \ref{symplecticklingen}. 
\begin{remark} Similarly to the symplectic case, unitary $V_1(p^r)$-level structures correspond to $p^r\O_F$-points. More precisely, let $\mathfrak{P}$ be an ideal of $\O_F$ lying above $p$ and define a $\mathfrak{P}^r$-point of an abelian scheme $A/S$ with $\O_F$-action $\iota:\O_F \longrightarrow End_S(A)$ to be an $\O_F$-linear monomorphism \[ (\O_F/\mathfrak{P}^r)_{/S} \hookrightarrow A[p^r].\] If $p$ is inert in $\O_F$, a $p^r\O_F$-point is just a point of exact order $p^r$. Indeed, if $P \in A(S)$ is a point of exact order $p^r$, it is killed by $\iota(a)$ for all $a \in p^r\O_F$ because $p^r\O_F=(p)^r$ and $\iota(a)P=\iota(a'p^r)P=\iota(a')[p^r]P=0$, for all $a=a'p^r \in p^r\O_F$. Thus, $P$ defines a monomorphism $$(\O_F/p^r\O_F)_{/S} \hookrightarrow A[p^r],$$ defined on each geometric fibres (corresponding to $s \to S$) by sending 1 to $P_s$ and identifying $\O_F/p^r\O_F$ with $\langle i(a)P_s \rangle_{a \in \O_F}$. On the other hand, any $p^r\O_F$-point is determined by the the point of exact order $p^r$ obtained by the image of $1$. In the case where $p$ splits in $F$, say $p=\mathfrak{P}\bar{\mathfrak{P}}$, then a $p^r\O_F$-point $P$ corresponds to $P_1$ and $P_2$ where \begin{itemize}
\item $P_1$ is a point of exact order $p^r$ which is killed by $\iota(a)$ for all $a \in \mathfrak{P}^r$, 
\item $P_2$ is a point of exact order  $p^r$ which is killed by $\iota(a)$ for all $a \in \bar{\mathfrak{P}}^r$.
\end{itemize}
\end{remark}
We can now state the unitary analogue of Lemma \ref{pointsandls}.
\begin{lemma} Let $(A,\lambda, \iota)$ be an abelian scheme of relative dimension 4 over an $F$-scheme $S$ with $\O_F$-module structure and principal polarisation. There is a bijection between unitary level $V_1(N)$-structures and $N\O_F$-points.
\end{lemma}
\begin{proof} 
It follows from a straightforward modification of the proof of Lemma \ref{pointsandls}.
\end{proof}

We are interested in showing a trace compatibility relation of the push-forward of the Eisenstein classes for $Y_H$ in the $p$-direction. \\  Let $V_N$ be the subgroup $V_1(N)V^{(N)} \subset G(\widehat{\Z})$, where $V^{(N)} \subset G(\widehat{\Z}^{(N)})$ is a sufficiently small open compact subgroup which satisfies the hypotheses of Lemma \ref{kisin} (for a suitable prime $q \nmid N$). 
\begin{remark} Note that the level subgroup $V_N$ is trivial outside a finite set of primes $\Sigma_{V_N} \ni q$. 
\end{remark}
Denote by $\phi_{N}$ the closed immersion \[Y_H(H \cap V_{N}) \longrightarrow Y_G(V_{N}).\]

\begin{definition}
Define ${_c\mathscr{Z}_{N}}$ to be the \'etale cohomology class defined by \[\phi_{N,*}(\Eis_{m,N}^2)\in H^{5}_{\et}(Y_G(V_{N}), \Z_{p}(3)).\] 
\end{definition}
We show that these classes are compatible under the trace $\operatorname{Tr}_{\pi_\ell'}$ (in \'etale cohomology) of the natural degeneracy map $\pi_\ell':Y_G(V_{N\ell}) \to Y_G(V_{N})$.
 
\begin{proposition}\label{klingenunitary}
Let $\ell,N$ be coprime with $c$ and let $\ell \not \in \Sigma_{V_N}$. Then, \[ Tr_{\pi_\ell'}({_c\mathscr{Z}_{N\ell}})= \begin{cases} {_c\mathscr{Z}_{N}}  & \text{ if }\ell \mid N; \\ 
\left( id - \varphi(\begin{smallmatrix} r I & 0 \\ 0 & r I \end{smallmatrix})^* \right){_c\mathscr{Z}_{N}} & \text{ if } \ell \nmid N; \end{cases}\]
 where $r$ denotes the inverse of $N$ modulo $\ell$.\end{proposition}
\begin{proof} 
It follows from Corollary \ref{rel1g}.
\end{proof}

While the compatibility of these classes in the tower of level subgroups $\{ V_N \}_N$ is a natural consequence of the trace compatibility relations of the Eisenstein classes, by using a more sophisticated method, in the next section we show how to obtain trace compatibility relations in a "two variable" tower of level subgroups $ \{V_{N,M} \}$.    

\subsection{Perturbing the embedding} 
 When $p \mid N$, we showed that the push-forward of Eisenstein classes defines an element \[{_c\mathscr{Z}_{Np^{\infty}}} \in \varprojlim_i H^{5}_{\et}(Y_G(V_{Np^i}), \Z_{p}(3)), \] where the limit is taken with respect to traces of the natural degeneracy maps $Y_G(V_{Np^{i+1}}) \to Y_G(V_{Np^i})$. In order to improve this result, it is necessary to enrich our  push-forward classes with extra structure. This is done by employing the action of $G(\Af)$ on $Y_G$, as we see below in Definition \ref{twistedembclass}. This idea has already been successfully used in the constructions of \cite{LLZ1}, \cite{LLZ2}, \cite{LSZ1}, and \cite{CR}. The method to prove extra trace compatibility relations used here is an adaptation to this setting of the one used in the proof of the vertical norm relation of the Beilinson-Flach Euler system as in \cite[Theorem 5.4.1]{KLZ} and its generalisation in \cite{LZvert}. \\ 
 In the rest of the section we suppose that $p \not \in \Sigma_{V_N}$ is inert or split in the imaginary quadratic field $F$ and that $N$ is an integer coprime with $p$. We now define the tower of level subgroups we are interested in.  
Let $\eta$ be the cocharacter of the maximal torus of $G$ defined by \[x \mapsto \left( \begin{smallmatrix} x^3 & {} & {} &   \\ & x^2 &   & \\  &   & x   &{} \\ & & & 1 \end{smallmatrix} \right) \] Let $\eta_p=\eta(p)\in G(\Q_p)$. It defines an open compact subgroup of $V_{N,1}' \subset V_{N}$ by requiring that $V_{N,1}'$ is the largest subgroup such that we have \[ \eta_p:Y_G(V_{N,1}') \longrightarrow Y_G(V_{N}). \] Moreover, let  $V_{N,p} \subset V_{N,1}'$ be the subgroup given by the intersection of  $V_{N,1}'$ with \[ \{g \in G(\Z_p):\quad g \equiv \left(\begin{smallmatrix} * & & & \\  & 1 & & \\ & & * & \\ & & & 1 \end{smallmatrix} \right) \text{  mod }p \}\]
Reiterating the procedure, we define subgroups of $G(\A_f)$   \begin{itemize} \item $V_{Np^n,p^m}':=V_{Np^n} \cap \eta_p^{m+1}V_{Np^n} \eta_p^{-(m+1)} \cap \{g \in G(\Z_p):\quad g \equiv \left(\begin{smallmatrix} * & & & \\  & 1 & & \\ & & * & \\ & & & 1 \end{smallmatrix} \right) \text{  mod }p^m \}$;
 \item $V_{Np^n,p^{m+1}}:=V_{Np^n,p^m}' \cap\{g \in G(\Z_p):\quad g \equiv \left(\begin{smallmatrix} * & & & \\  & 1 & & \\ & & * & \\ & & & 1 \end{smallmatrix} \right) \text{  mod }p^{m+1} \}$. 
\end{itemize}
Then, inductively define $V_{Np^n,Mp^{m}}$ for $M$ coprime to $\Sigma_{V_N}$.
\begin{remark}
Concretely, for $n>3m$, $V_{Np^n,Mp^{m}}$ has component at $p$ \[ \bigg\{ g \in G(\Z_p) : g \equiv I  \text{ mod } \left[ \begin{smallmatrix} 1 & p^{m} & p^{2m} &  p^{3m} \\ p^n & p^{m} & p^{m} & p^{2m} \\ p^n & p^{m} & 1 & p^{m} \\ p^n & p^n & p^n & p^{n}  \end{smallmatrix}  \right] \O_F \bigg\}. \] 
\end{remark}
 For $u\in G(\Af)$, let $\phi_V^u$ be the composition  \[ u \circ \phi_{uVu^{-1}}: Y_H(H \cap uVu^{-1}) \longrightarrow Y_G(uVu^{-1})\longrightarrow Y_G(V), \] where the second arrow is given by right multiplication by $u$. 
 In particular, consider $u \in G(\Af)$ such that \begin{enumerate} \item $H \cap uV_{Np^n,Mp^m}u^{-1} \subset U_{Np^n},$ \item $\phi_{uV_{Np^n,Mp^m}u^{-1}}$ is a closed immersion. 
 \end{enumerate} 
 \begin{remark} The two conditions are satisfied by $u \in V_{Np^n}$  with trivial components at places in $\Sigma_{V_N}$. 
 \end{remark}
 For such $u$, we consider the pull-back of $\Eis_{Np^n}$ to the cohomology group of $Y_H(H \cap uV_{Np^n,Mp^m}u^{-1})$ and we again denote by $\Eis_{Np^n}$.

  \begin{definition}\label{twistedembclass}
For $u\in G(\Af)$ satisfying the above conditions, define \[ _c\mathscr{Z}_{Np^n,Mp^m,u}:= \phi_{V_{Np^n,Mp^m,*}}^u ( \Eis_{Np^n}) \in H_{\et}^{5}(Y_G(V_{Np^n,Mp^m}) ,\Z_{p}(3)). \]
\end{definition}
We can finally state the main theorem of the section.

\begin{theorem}\label{finalcongruences} Suppose $n\geq 3m+3$ and let $m \geq 1$. \begin{enumerate} 
\item Let $\pi_p: Y_G(V_{Np^{n+1},Mp^{m}}) \longrightarrow  Y_G(V_{Np^{n},Mp^{m}})$ be the natural degeneracy map, then 
\[ \operatorname{Tr}_{\pi_p}({}_c\mathcal{Z}_{Np^{n+1},Mp^m,u})={}_c\mathcal{Z}_{Np^n,Mp^m,u}.\]
\item Let $\mathcal{U}_p$ be the Hecke operator associated to $\eta_p$ and let $\operatorname{Tr}_{f_p}$ be the trace map associated to \[f_p:Y_G(V_{Np^n,Mp^{m+1}}) \longrightarrow Y_G(V_{Np^n,Mp^{m}}') \longrightarrow Y_G(V_{Np^n,Mp^{m}}),\] where the first arrow is the natural degeneracy map and the second is right multiplication by $\eta_p$. There exists $u \in G(\Af)$ such that 
\[ \operatorname{Tr}_{f_p}({}_c\mathcal{Z}_{Np^n,Mp^{m+1},u})=\mathcal{U}_p\cdot {}_c\mathcal{Z}_{Np^n,Mp^m,u}.\]
\end{enumerate}
\end{theorem} 
\begin{remark}\label{commentisuteoprinc} \leavevmode
\begin{enumerate}
\item The Hecke operator $\mathcal{U}_p$ is defined as the correspondence \[ \xymatrix{ Y_G(V_{Np^n,Mp^{m}}') \ar[d]_{\bar{pr}} \ar[dr]^{\eta_p} & \\ Y_G(V_{Np^n,Mp^{m}}) \ar@{.>}[r]^{\mathcal{U}_{p}} & Y_G(V_{Np^n,Mp^{m}}),}  \] where the vertical arrow $\bar{pr}$ is the natural degeneracy map. 
\item After applying the ordinary idempotent $e_{\eta_p}:= \lim_{k\to \infty} \mathcal{U}_p^{k!}$, we get a "neat" compatibility in the second variable (Definition \ref{IWelt}).  
\item The proof of Theorem \ref{finalcongruences}$(1)$ is Proposition \ref{klingenunitary}.
\item The $u$ of Theorem \ref{finalcongruences}$(2)$ does not depend on either $n$ or $m$ and it is not unique (see section 5.4 and the appendix for further explanations).
\end{enumerate}
\end{remark}

Since the nature of these statements is local at $p$,  there is no harm in assuming $M=1$. \\  Theorem \ref{finalcongruences}$(2)$ follows from the following.
\begin{proposition}\label{propcartesian}
There exists an element $u \in G(\Af)$ such that the commutative diagram 
\[ \xymatrix{ & & & & Y_G(V_{Np^n,p^{m+1}}) \ar[d]^{pr}  \\ 
Y_H(uV_{Np^n,p^{m+1}}u^{-1} \cap H)  \ar[urrrr]^-{\phi_{V_{\tiny{Np^n,p^{m+1}}}}^u}  \ar[rrrr]^-{pr \circ \phi_{V_{\tiny{Np^n,p^{m+1}}}}^u}  \ar[d]_{\pi} & & & &  Y_G(V_{Np^n,p^m}') \ar[d]^{\bar{pr}} \\
Y_H(uV_{Np^n,p^m}u^{-1} \cap H)  \ar[rrrr]^-{ \phi_{V_{\tiny{Np^n,p^m}}}^u} & & &  &Y_G(V_{Np^n,p^m})} \] 
has Cartesian bottom square.
\end{proposition}
See Appendix A for a proof of Proposition \ref{propcartesian}.
\begin{proof}[Proof of Theorem \ref{finalcongruences}$(2)$]
Let $u$ be the matrix which appears in Proposition \ref{propcartesian}. From the compatibility of pull-backs and push-forwards in Cartesian diagrams, we get \[\operatorname{Tr}_{pr}({}_c\mathcal{Z}_{p^n,p^{m+1},u})=\bar{pr}^{*}({}_c\mathcal{Z}_{p^n,p^{m},u}). \] Applying the trace of $\eta_p:Y_G(V_{Np^n,p^m}')\longrightarrow Y_G(V_{Np^n,p^m})$, we have \begin{align*} \operatorname{Tr}_{\eta_p}( \operatorname{Tr}_{pr}({}_c\mathcal{Z}_{p^n,p^{m+1},u}))&=\operatorname{Tr}_{\eta_p}(\bar{pr}^{*}({}_c\mathcal{Z}_{p^n,p^{m},u})) \\ &= \mathcal{U}_p\cdot {}_c\mathcal{Z}_{p^n,p^{m},u}, \end{align*} where the last equality follows from the very definition of $\mathcal{U}_p$ in Remark \ref{commentisuteoprinc} (1) as the correspondence $\operatorname{Tr}_{\eta_p} \circ \bar{pr}^{*}$; this is the desired formula since $\operatorname{Tr}_{\eta_p} \circ \operatorname{Tr}_{pr}=\operatorname{Tr}_{f_p}.$ 
\end{proof}

\subsection{Projection to the ordinary part} Using Theorem \ref{finalcongruences}, we define a limiting element where both $n$ and $m$ go to $\infty$.
If we substitute the tower of level subgroups $K_{Np^n,p^m}:=\eta_p^{-m}V_{Np^n,p^m}\eta_p^m$ for $V_{Np^n,p^m}$ and the class \[ _c\mathfrak{Z}_{Np^n,p^m,u}:=(\eta_p^{m})_*({}_c\mathcal{Z}_{Np^n,p^m,u})\in H_{\et}^{5}(Y_G(K_{Np^n,p^m}) ,\Z_{p}(3)) \] for ${}_c\mathcal{Z}_{Np^n,p^m,u}$, Theorem \ref{finalcongruences} gives a trace compatibility relation with respect to the natural degeneracy maps for both $n$ and $m$ varying. This follows from having the commutative diagram \[\xymatrix{ Y_G(V_{Np^n,p^{m+1}}) \ar[rr]^{\eta_p^{m+1}} \ar[d]_{f_p} &&Y_G(K_{Np^n,p^{m+1}}) \ar[d] \\ Y_G(V_{Np^n,p^m}) \ar[rr]^{\eta_p^{m}} & & Y_G(K_{Np^n,p^m}), } \] where right vertical map is the natural degeneracy map. Set \[ H_{\et}^{5}(Y_G(K_{Np^{\infty},p^{\infty}}) ,\Z_{p}(3)):= \varprojlim_{n,m} H_{\et}^{5}(Y_G(K_{Np^n,p^m}) ,\Z_{p}(3)).\]
The ordinary idempotent $e_{\eta_p}:= \lim_{k\to \infty} \mathcal{U}_p^{k!}$ acts on it.  \begin{definition}\label{IWelt} We define \[  _c\mathfrak{Z}_{N,u}^{\text{ord}}:=(\mathcal{U}_{p}^{-m}\cdot  e_{\eta_p}({}_c\mathfrak{Z}_{Np^n,p^m,u}))_{n,m\geq 1} \in  e_{\eta_p}\cdot H_{\et}^{5}(Y_G(K_{Np^{\infty},p^{\infty}}) ,\Z_{p}(3)). \]
 \end{definition}
The class $ _c\mathfrak{Z}_{N,u}^{\text{ord}}$ is well-defined, thanks to the trace compatibility relations of Theorem \ref{finalcongruences}.
 \subsection{Final remarks}
 In \cite{LLZ1}, \cite{LLZ2}, \cite{LSZ1}, and \cite{CR}, the trace compatibility relations obtained by varying level subgroups $ V_{Np^n,p^m}$ with respect to $m$ has a primary role in proving the vertical Euler system norm relations in the $p$-cyclotomic tower. Theorem \ref{finalcongruences} does not give any result in this direction. In this section, we describe the nature of the obstruction that we encounter when trying to prove cyclotomic norm relations using this method. This informal discussion is much inspired by the work \cite{LZvert}, which treats an axiomatisation of the technique used in \emph{op.cit}. 
\\ 

  Ideally we would have projections \[Y_G(V_{Np^n,p^m})\longrightarrow Y_G(V_{Np^n})\times_{\Q} \operatorname{Spec}\Q(\zeta_{p^m}) \] and we would read the compatibility under $\operatorname{Tr}_{f_p}$ of Theorem \ref{finalcongruences} as one under the trace map associated to the natural degeneracy map $ \operatorname{Spec}\Q(\zeta_{p^{m+1}})\to \operatorname{Spec}\Q(\zeta_{p^m})$. Unfortunately, \[ \nu(V_{Np^n,p^m})\subset \widehat{\Z}^*\] has  $p$-part equal to $\Z_p^*$, thus $Y_G(V_{Np^n,p^m})$  does not surjects onto $Y_G(V_{Np^n})\times_{\Q} \operatorname{Spec}\Q(\zeta_{p^m})$ since it does not have enough connected components at $p$.\\  
Motivated by the Euler system constructions mentioned above, we could try to modify the tower of subgroups $\{V_{Np^n, p^m}\}$ by defining \[ \tilde{V}_{Np^n, p^m}:=V_{Np^n} \cap \eta_p^{m}V_{Np^n} \eta_p^{-m} \cap \{g \in G(\Z_p):\quad \nu(g) \equiv 1 \text{ (mod }p^m )\}, \] instead of intersecting $V_{Np^n} \cap \eta_p^{m}V_{Np^n} \eta_p^{-m}$ with \[\{g \in G(\Z_p):\quad g \equiv \left(\begin{smallmatrix} * & & & \\  & 1 & & \\ & & * & \\ & & & 1 \end{smallmatrix} \right) \text{  mod }p^m \}.\] 
Then, the corresponding Shimura variety $Y_G(\tilde{V}_{Np^n,p^m})$ has enough connected components and it is reasonable to ask if it is possible to find $u\in G(\Af)$ such that the diagram \[ \xymatrix{ & & & & Y_G(\tilde{V}_{Np^n,p^{m+1}}) \ar[d]  \\ 
Y_H(u\tilde{V}_{Np^n,p^{m+1}}u^{-1} \cap H)  \ar[urrrr]^-{\phi_{\tilde{V}_{\tiny{Np^n,p^{m+1}}}}^u}  \ar[rrrr]^-{pr \circ \phi_{\tilde{V}_{\tiny{Np^n,p^{m+1}}}}^u}  \ar[d] & & & &  Y_G(\tilde{V}_{Np^n,p^m}') \ar[d] \\
Y_H(u\tilde{V}_{Np^n,p^m}u^{-1} \cap H)  \ar[rrrr]^-{ \phi_{\tilde{V}_{\tiny{Np^n,p^m}}}^u} & & &  &Y_G(\tilde{V}_{Np^n,p^m})} \]  has Cartesian bottom square. Crucially, $u\in G(\Af)$ has to be chosen so that the morphism \[ pr \circ \phi_{\tilde{V}_{\tiny{Np^n,p^{m+1}}}}^u \] is a closed immersion. This boils down to showing that, for any $m\geq 1$, \begin{align}\label{wish} u\tilde{V}_{Np^n,p^{m+1}}u^{-1} \cap H=u\tilde{V}_{Np^n,p^{m}}'u^{-1} \cap H.\end{align} 
How does one determine $u$ such that equality \eqref{wish} is satisfied? Our choice of the co-character $\eta$ determines a parabolic and its opposite of $G$, which are respectively the upper and lower-triangular Borels $B_G$ and $\bar{B}_G$ (indeed,  conjugation by powers of $\eta_p$ induces congruences modulo powers of $p$ for upper triangular entries of the elements in $V_{Np^n,p^m}$). 
The equality \eqref{wish} follows (by reducing modulo $p^{m+1}$) from the condition \begin{align}\label{wish2}  \text{Klin}^o_{H} \cap u\bar{B}_Gu^{-1}\subset \operatorname{Sp}_4,\end{align} where $\text{Klin}^o_{H}$ is the $p$-part of the stabiliser of the Eisenstein class for $H$ \[ \Eis_{m,Np^{\infty}}^2 \in H^{3}_{\et}(Y_{H}(U_{Np^{\infty}}),\Z_p(2)).\]  In other words, it denotes the subgroup over $\Z_p$ of the Klingen parabolic of $H$ of matrices of the form \[ \left(\begin{smallmatrix} \ast & \ast & \ast & \ast \\ 0 & \ast & \ast & \ast \\ 0 & \ast & \ast & \ast \\ 0 & 0 & 0 & 1 \end{smallmatrix}\right). \]

\begin{remark} Note that $\text{Klin}^o_{H} \cap u\bar{B}_Gu^{-1}$ is the stabiliser of the $\text{Klin}^o_{H}(\Z_p)$-orbit $u\bar{B}_G$ in the flag variety $G/\bar{B}_G$.
\end{remark}  
We cannot find $u$ such that its stabiliser is contained in $\operatorname{Sp}_4$. For instance, let $p$ be split in $F$, so that $G(\Q_p)$ is isomorphic to $\GL_4(\Q_p) \times \Gm(\Q_p)$. The lower-triangular Borel $\bar{B}_G$ has co-dimension 6 in $G$, while $\text{Klin}^o_{H}$ has dimension 7, thus $u$ satisfies \eqref{wish2} if the image under the multiplier $\mu$ of a space of dimension bigger or equal than 1 is trivial. For sufficiently generic $u$, conjugation by $u$ rearranges the entries of matrices in $\bar{B}_G$, but the condition that they need to lie in $\text{Klin}^o_{H}$ does not give enough equations to force these matrices to have multiplier one.  

\begin{remark} This is precisely why we define the tower of subgroups $V_{Np^n,p^m}$ by intersecting it with \[\{g \in G(\Z_p):\quad g \equiv \left(\begin{smallmatrix} * & & & \\  & 1 & & \\ & & * & \\ & & & 1 \end{smallmatrix} \right) \text{  mod }p^m \}.\] Indeed, the calculation underlying Proposition \ref{propcartesian} shows that the stabiliser of the open $\text{Klin}^o_{H}$-orbit $u\bar{B}_G$ is one dimensional and it is isomorphic to the one dimensional subgroup  \[ \left \{ \left(\begin{smallmatrix} x & & & \\  & 1 & & \\ & & x & \\ & & & 1 \end{smallmatrix} \right)\right \} \] of the maximal torus of $H$.
\end{remark}
\begin{remark}
This phenomenon occurs in several push-forward constructions from an Eisenstein class for $\GSp_4$. For instance, an analogous of Theorem \ref{finalcongruences} works in the case where we consider \[ H\times_{\mu,\operatorname{det}}\GL_2 \hookrightarrow \GSp_6 \] and the push-forward of the pull-back of $\Eis_{m,Np^{n}}^2$ along the diagram \[ \xymatrix{ Y_{H}(U_{Np^n}) & Y_H(L_{Np^n} \cap H) \ar[l] \ar[r] & Y_{\GSp_6}(L_{Np^n})},\] for sufficiently nice level subgroup $L_{Np^n} \subset \GSp_6(\Af)$. In this case, there is $u \in \GSp_6$ such that the stabiliser of its $\text{Klin}^o_{H}\times_{\mu,\operatorname{det}}\GL_2$-orbit in the flag variety $\GSp_6/\bar{B}_{\GSp_6}$ is isomorphic to the one dimensional subgroup $\{ \operatorname{diag}(x,1,x,1,x,1) \}$, and for which the analogue of Proposition \ref{propcartesian} holds.  For instance, a representative $u$ of this orbit can be taken as \[\left(\begin{smallmatrix} 1 & 2& 1 & -1 & 1 &1  \\ 0 & 3 & 2 & 0 & 1 & -1  \\ 0 & 2 & 2 & 1 & 0 & -1  \\ 0 & 1 & 1 & 2 & -1 &0  \\ 0 & 0 & 0 & -1 & 1 & -1  \\ 0 & 0 & 0 & 0 & 0 & 1 \end{smallmatrix} \right). \]
\end{remark}

It remains an interesting open problem to understand the connection of the construction of the \'etale classes \[ _c\mathscr{Z}_{Np^n,Mp^m,u} \in H_{\et}^{5}(Y_G(V_{Np^n,Mp^m}) ,\Z_{p}(3)) \]
with values of the wedge square product $L$-function of automorphic representations of $G$ and we intend to tackle the question in a future project.
\appendix
\section{Proof of Proposition \ref{propcartesian}}

In the following, we show that there exists $u \in G(\Af)$ such that the diagram of Proposition \ref{propcartesian} 
has Cartesian bottom square. This follows from showing that there exists $u \in G(\Af)$ such that \begin{enumerate}
\item $pr \circ \phi_{V_{\tiny{Np^n,p^{m+1}}}}^u$ is a closed immersion, i.e. \[uV_{Np^n,p^{m+1}}u^{-1} \cap H=uV_{Np^n,p^{m}}'u^{-1} \cap H.\]
\item the degrees of $\pi$ and $\bar{pr}$ agree, i.e. \[ [V_{Np^n,p^{m}}:V_{Np^n,p^{m}}']=[uV_{Np^n,p^{m}}u^{-1} \cap H:uV_{Np^n,p^{m+1}}u^{-1} \cap H].\]
\end{enumerate}
Here $n$ is always assumed to be bigger than $3m+3$.  
We treat the cases of split and inert $p$ in $F$ separately. 
\subsection{The split case}
Let $p$ be split in $F$ and denote $\Z_p \otimes_{\Z} \O_F$ by $\O_{F,p}$. Note that there is a $\Z_p$-algebra isomorphism \[ i_p: \O_{F,p} \longrightarrow \Z_p \times \Z_p.\] If $i_1,i_2$ denote the two distinct embeddings of $F$ into $\Q_p$, then \[i_p:a \otimes b \mapsto (i_1(b)a,i_2(b)a).\] Under this isomorphism, \[ i_p(a \otimes \bar{ b})=(i_2(b)a,i_1(b)a).\] This isomorphism induces one between $G(\Z_p)$ and $$\left \{ (M,N,a) \in GL_4(\Z_p)\times GL_4(\Z_p)\times \Z_p^* : N^t J M=a J \right \}.$$ Moreover, the map $(M,N) \mapsto (M,a)$  defines an isomorphism between $G(\Z_p)$ and $GL_4(\Z_p) \times \Gm(\Z_p)$ and $H(\Z_p)$ embeds into $GL_4(\Z_p) \times \Gm(\Z_p)$ via $M\mapsto (M,\mu(M)).$
We claim that \[u:=\left( \left( \begin{smallmatrix} 1 & & -1 & \\  & 1 & & \\ & & 1 & \\ & & & 1 \end{smallmatrix} \right),1\right) \in GL_4(\Z_p) \times \Gm(\Z_p)\] satisfies the properties (1),(2) listed above.
\begin{lemma}\label{closedimmsplit} We have
$uV_{Np^n,p^{m+1}}u^{-1} \cap H=uV_{Np^n,p^{m}}'u^{-1} \cap H$.
\end{lemma}
\begin{proof}
We want to show that if $g=(g_p)_p \in V_{Np^n,p^{m}}'$ is such that $ugu^{-1} \in H$, then $g \in V_{Np^n,p^{m+1}}$. Since this is a local statement at $p$, it is enough to verify that if $g_p=\left( \left( \begin{smallmatrix} A & B \\  C & D \end{smallmatrix} \right),\alpha \right)$ satisfies $ug_pu^{-1} \in H(\Z_p)$, then \[ g_p \equiv \left( \left( \begin{smallmatrix} x & & & \\  & 1 & & \\ & & x & \\ & & & 1 \end{smallmatrix} \right), x \right) \text{ (mod }p^{m+1}).\] 
Reducing modulo $p^{m+1}$, we get $ug_pu^{-1}\equiv \left( \left( \begin{smallmatrix} a_1 & -c & a_1-d_1 & 0 \\  0 & a_4 & 0 & 0 \\ 0 & c & d_1 & 0 \\ 0 & 0 & 0 & 1 \\ \end{smallmatrix} \right),\alpha \right)$, which is symplectic if \begin{align*}  \left( \begin{smallmatrix} 0 & a_1 \\ a_4d_1 & -c \end{smallmatrix} \right) &= \left( \begin{smallmatrix} 0 & \alpha \\  \alpha & 0 \end{smallmatrix} \right), \\ \left( \begin{smallmatrix} 0 & a_1-d_1 \\ 0 & 0 \end{smallmatrix} \right) &= \left( \begin{smallmatrix} 0 & 0 \\  a_1-d_1 & 0 \end{smallmatrix} \right).
\end{align*}
Thus, $a_1-d_1\equiv a_4-1 \equiv c \equiv 0$ (mod $p^{m+1}$), which implies that \[ g_p \equiv \left( \left( \begin{smallmatrix} a_1 & & & \\  & 1 & & \\ & & a_1 & \\ & & & 1 \end{smallmatrix} \right), a_1 \right) \text{ (mod }p^{m+1}).\]
\end{proof}
 We are left to check that 
 \begin{lemma}\label{splitdeg} We have
 \[ [V_{Np^n,p^{m}}:V_{Np^n,p^{m}}']=[uV_{Np^n,p^{m}}u^{-1} \cap H:uV_{Np^n,p^{m}}'u^{-1} \cap H].\]
 \end{lemma}
 \begin{proof}  
Note that $[V_{Np^n,p^{m}}:V_{Np^n,p^{m}}']=p^{10}$, since a left coset of representatives is given by
\[ \sigma_{\underline{v}} =\left( \left( \begin{smallmatrix} 1 & p^m k_1 & p^{2m} r_1 & p^{3m} r_2  \\ & 1 & p^m r_3 & p^{2m}r_4  \\  &  & 1 & p^m k_2  \\    &  & & 1 \end{smallmatrix} \right), 1 \right), \]
where for each vector $\underline{v}\in \Z/ p^3\Z \times (\Z/ p^2\Z)^{2} \times (\Z/ p\Z)^{3}$ we consider one and only one lift 
\[(r_2,r_1,r_4,k_1,r_3,k_2)\in \Z_p^{10}.\]

Now, recall that the $p$-component of $uV_{Np^n,p^{m}}u^{-1} \cap H$ is isomorphic to the subgroup of $G(\Z_p)$ given by elements $(g,\alpha)$ such that $u(g,\alpha) u^{-1} \in H$ and
\[ g \equiv  \left( \begin{smallmatrix} \alpha & & & \\  & 1 & & \\ & & \alpha & \\ & & & 1 \end{smallmatrix} \right)  \text{ mod } \left[ \begin{smallmatrix} p^m & p^{m} & p^{2m} &  p^{3m} \\ p^n & p^{m} & p^{m} & p^{2m} \\ p^n & p^{m} & p^m & p^{m} \\ p^n & p^n & p^n & p^{n}  \end{smallmatrix}  \right]. \] 
Moreover, from Lemma \ref{closedimmsplit} we have $uV_{Np^n,p^{m}}'u^{-1} \cap H=uV_{Np^n,p^{m+1}}u^{-1} \cap H$, hence \[ [uV_{Np^n,p^{m}}u^{-1} \cap H:uV_{Np^n,p^{m}}'u^{-1} \cap H] = [uV_{Np^n,p^{m}}u^{-1} \cap H:uV_{Np^n,p^{m+1}}u^{-1} \cap H].\]
We claim that a system of coset representatives is given by a subset of the set of elements \[ \sigma_{\underline{v}}' = \left( \left( \begin{smallmatrix} 1+p^m s_1 & p^m k_1 & p^{2m} r_1 & p^{3m} r_2  \\ & 1+p^m s_2 & p^m r_3 & p^{2m}r_4  \\  & p^m t & 1 & p^m k_2  \\    &  & & 1 \end{smallmatrix} \right), 1+p^m s_1 \right) \in V_{Np^n,p^{m}} \]
where for each vector $\underline{v} \in \Z/ p^3\Z \times (\Z/ p^2\Z)^{2} \times (\Z/ p\Z)^{6}=:S$ we consider one and only one lift 
\[(r_2,r_1,r_4,k_1,r_3,k_2,s_1 , s_2 , t )\in \Z_p^{13}.\]
More precisely, a system of left coset representatives is $\{ u({\sigma_{\underline{w}}'})^{-1}u^{-1} \}_{\underline{w} \in W}$, where $W \subset S$, which is determined by the symplectic conditions for $u{\sigma_{\underline{w}}'}u^{-1}$ \emph{modulo (powers of) $p$}, 
is defined to be the subset of cardinality $p^{10}$ of elements of the form \[ (r_2,r_1,r_4,k_1,r_3,k_2,0 , 0 , k_1 + k_2 ) \in S.\]  

To prove our claim we need to show that for any $g=\left( \begin{smallmatrix} a_1 & a_2 & b_1 & b_2 \\ 0 & a_4 & b_3 & b_4 \\0 & c_2 & d_1 & d_2 \\ 0 & 0 & 0 & 1 \end{smallmatrix} \right) \in V_{Np^n,p^{m}}$ such that $ugu ^{-1} \in H$, there exists $\sigma_{\underline{v}}'$ with $\underline{v} \in W$ such that \[ u\sigma_{\underline{v}}'  g u^{-1} \in  uV_{Np^n,p^{m+1}}u^{-1} \cap H. \] 
This boils down to solving the system of equations 
\[ \begin{array}{l l}
{\left\{ 
	\begin{array} {l l} 
	b_2 + p^m k_1 b_4 + p^{2m} r_1 d_2 + p^{3m} r_2 \equiv 0 \;\; [p^{3m+3}] \\
	b_1 + p^m k_1 b_3 + p^{2m} r_1 d_1  \equiv 0 \;\; [p^{2m+2}] \\
	b_4 + p^m r_3 d_2 + p^{2m} r_4 \equiv 0 \;\; [p^{2m+2}] \\
	a_2 + p^m k_1 a_4   \equiv 0 \;\; [p^{m+1}] \\
	b_3 + p^m r_3 \equiv 0 \;\; [p^{m+1}] \\
	p^m (k_1 + k_2 ) b_4  + d_2 + p^{m} k_2 \equiv 0 \;\; [p^{m+1}] 
	\end{array}
	\right. } \\
{\left\{ 
	\begin{array} {l l} 
	a_4 + p^m r_3 c_2 \equiv 1 \;\; [p^{m+1}] \\
	p^m (k_1 + k_2 ) a_4  + c_2 \equiv 0 \;\; [p^{m+1}] \\
	p^m (k_1 + k_2 ) b_3 + d_1 \equiv a_1 \;\; [p^{m+1}] \\
	\end{array}
	\right.} \\
\end{array} \\
\]
From the first system of equations, we determine the values of $r_2,r_1,r_4,k_1,r_3,$ and $k_2$ since $d_1$ and $a_4$ are invertible modulo $p$. In particular, the fourth and sixth equations give \[ p^m(k_1+k_2) \equiv - a_4^{-1} a_2 - d_2 \;\; [p^{m+1}].\] 
Thus, the second system of equations reduces to \[ {\left\{ 
	\begin{array} {l l} 
	a_4  \equiv 1 \;\; [p^{m+1}] \\
	 - a_2 - d_2 + c_2 \equiv 0 \;\; [p^{m+1}] \\
	d_1 \equiv a_1 \;\; [p^{m+1}] \\
	\end{array}
	\right.}\]
 and it is redundant. Indeed, unfolding the symplectic conditions of $ugu^{-1}$ modulo $p^{m+1}$, we get 
 
\[ {\left\{ \begin{array} {l l} 
	a_1  \equiv a_4 d_1 \;\; [p^{m+1}] \\
	a_2 - c_2 + a_4 d_2 \equiv 0 \;\; [p^{m+1}] \\
	d_1 b_4 \equiv b_2 + a_1 - d_1 + d_2 b_3 \;\; [p^{m+1}] \\
	\end{array}
	\right.}  \]	
which gives 
\[ {\left\{ \begin{array} {l l} 
	a_4  \equiv 1 \;\; [p^{m+1}] \\
	c_2 \equiv a_2 + d_2   \;\; [p^{m+1}] \\
	d_1 \equiv  a_1 \;\; [p^{m+1}] \\
	\end{array}
	\right.}  \]	
	since $b_4,b_2,b_3 d_2 \equiv 0$ modulo $p^{m+1}$. Thus,  we conclude that 
 \[ [V_{Np^n,p^{m}}:V_{Np^n,p^{m}}']=[uV_{Np^n,p^{m}}u^{-1} \cap H:uV_{Np^n,p^{m}}'u^{-1} \cap H]=p^{10}.\]
\end{proof}

\subsection{The inert case}
Let $F_p$ be the $p$-adic completion of $F$ at $p$; it is an extension of degree 2 over $\Q_p$ and denote by $\bar{\bullet}$ the non-trivial automorphism in the Galois group. Let $e\in \O_{F_p}$ be a generator of $\O_F/(p\O_F+\Z)$ and consider \[ u =\left( \begin{smallmatrix} 1 &  & e & \\  & 1 & & \bar{e} \\ &  & 1 &  \\ & & & 1 \end{smallmatrix} \right)\in  G(\Z_p).\]
We claim that $u$ satisfies (1),(2).
\begin{lemma}\label{firstinert} We have
$uV_{Np^n,p^{m+1}}u^{-1} \cap H=uV_{Np^n,p^{m}}'u^{-1} \cap H$.
\end{lemma}
\begin{proof}
 As for Lemma \ref{closedimmsplit}, it suffices to show that if $g=(g_p)_p \in V_{Np^n,p^{m}}'$ is such that $ug_pu^{-1} \in H$, then $g \in V_{Np^n,p^{m+1}}$.  Note that the condition $ug_pu^{-1} \in H(\Z_p)$ is equivalent to asking that $ug_pu^{-1}$ has entries in $\Z_p$. Modulo $p^{m+1}\O_F$, we have \[ ug_pu^{-1} \equiv \left( \begin{smallmatrix} a_1 & e c_2 & e(d_1-a_1) & e\bar{e}c_2 \\  0 & a_4 & 0 & \bar{e}(1-a_4)  \\ 0 &  c_2 & d_1 & \bar{e}c_2 \\ 0 & 0 & 0 & 1 \end{smallmatrix} \right).\] Thus, $ug_pu^{-1} \in H(\Z_p)$ implies  \[a_1-d_1\equiv a_4-1 \equiv c_2 \equiv 0  \text{ (mod }p^{m+1}),\] hence $g \in V_{Np^n,p^{m+1}}$. 
\end{proof}
\begin{lemma}\label{deginert} We have
 \[ [V_{Np^n,p^{m}}:V_{Np^n,p^{m}}']=[uV_{Np^n,p^{m}}u^{-1} \cap H:uV_{Np^n,p^{m}}'u^{-1} \cap H].\]
 \end{lemma}
\begin{proof}
As in the split case, we have \[[V_{Np^n,p^{m}}:V_{Np^n,p^{m}}']= p^{10}.\] Indeed, a system of coset representatives is given by \[ \sigma_{\underline{v}} =\left( \begin{smallmatrix} 1 & p^m k_1 & p^{2m} r_1 & p^{3m} r_2  \\ & 1 & p^m r_3 & p^{2m} r_4  \\  &  & 1 & - p^m \bar{k}_1  \\    &  & & 1 \end{smallmatrix} \right) \; : \; \; {\left\{ 
	\begin{array} {l l} 
	r_2- k_1 r_4  \in \Z_p, \\
	 r_4= \bar{r}_1 - \bar{k}_1 r_3,
		\end{array}
	\right.} \]
where for each vector $\underline{v} \in \Z/ p^3\Z \times \O_F/ p^2\O_F \times \O_F/ p\O_F \times \Z/ p\Z$ we consider one and only one lift 
\[(\tilde{r}_2,r_1,k_1,r_3)\in \Z_p \times \O_{F_p}^{2} \times \Z_p,\]  so that $r_2=\tilde{r}_2+k_1( \bar{r}_1 - \bar{k}_1 r_3)$. \\ 
The calculation of $[uV_{Np^n,p^{m}}u^{-1} \cap H:uV_{Np^n,p^{m}}'u^{-1} \cap H]=p^{10}$ is very similar to the one in Lemma \ref{splitdeg}. Here, a system of left coset representatives is formed by elements $u\sigma_{\underline{w}}'u^{-1} \in H$, where
\[\sigma_{\underline{w}}':=\left( \begin{smallmatrix} 1+p^m s_1 & p^m k_1 & p^{2m} r_1 & p^{3m} r_2  \\ & 1+p^m s_2 & p^m r_3 & p^{2m} r_4  \\  & p^mt & 1+ p^m s_3 &  p^m k_2  \\    &  & & 1 \end{smallmatrix} \right) \in V_{Np^n,p^{m}} \] where for each vector
 \[ \underline{w}\in (\Z/ p\Z)^5\times \Z/ p^3 \Z \times  (\O_F/ p\O_F)^2 \times  (\O_F/ p^2\O_F)^{2} \] we consider one and only one lift 
\[(s_1,s_2, s_3, r_3 , t , r_2 , k_1 , k_2 , r_1 , r_4 )\in \Z_p^6 \times \O_{F_p}^{4} \] 
subject to conditions 
\[ 
{\left\{ 
	\begin{array} {l l} 
	s_{1}= s_{2} +  s_3 + p^m( s_2 s_3 - t r_3 ), \\
	\bar{k}_{1} = - k_2 - p^m(k_2 s_2 + p^m t r_4), \\
	\bar{k}_{2}  r_4 \in \Z_p, \\
	\bar{r}_1 =(1+ p^m s_3) - k_2 r_3, \\ 
	k_2-\bar{e}t \in \Z_p, \\ 
	p^m r_4 + \bar{e} s_2 \in \Z_p, \\
	
	\end{array}
	\right. } \]
Looking carefully at the system, we can recover $s_1,s_2,k_1 , k_2,$ and $r_1$ in terms of $ s_3, r_3, t , r_2,$ and $r_4$.
Indeed, $k_2$ is obtained from equations 3 and 5, while the values $s_1,s_2,k_1 ,$ and $r_1$ are determined respectively by equations 1, 6 , 2, and 4.
Thus, a system of left coset representatives is given by $\{ u\sigma_{\underline{w}}'u^{-1}\}_{\underline{w}}$, where $\underline{w}$ ranges through all the vectors in $(\Z/ p\Z)^3 \times \Z/p^3 \Z \times (\O_F/ p^2\O_F)$.



\end{proof}

\bibliographystyle{alpha}

\bibliography{normcompatibleclassesforgu}

\end{document}